\documentclass[11pt]{article}
\title{\vspace{-1.0cm}\textbf{Scaling limits and fluctuations for random growth under capacity rescaling}}
\author{\textbf{George Liddle\footnote{g.liddle@lancaster.ac.uk}, Amanda Turner}\footnote{a.g.turner@lancaster.ac.uk}}

\usepackage{geometry}
\newgeometry{vmargin={35mm}, hmargin={35mm,30mm}}
\usepackage{amsthm}
\usepackage{amsmath}
\usepackage{color}
\usepackage{mathrsfs}
\usepackage{eufrak}
\usepackage{ dsfont }
\usepackage{graphicx}
\usepackage{bbold}
\usepackage{xcolor}
\usepackage{enumerate}

\newtheorem{theorem}{Theorem}[section]
\newtheorem{lemma}[theorem]{Lemma}
\newtheorem{proposition}[theorem]{Proposition}
\newtheorem{corollary}[theorem]{Corollary}
\newtheorem{definition}[theorem]{Definition}

\newtheorem*{thm1}{Theorem 1.1}
\newtheorem*{thm2}{Theorem 1.2}
\newtheorem*{thm3}{Theorem 1.3}

\date{}

\begin{document}
\maketitle
\begin{abstract}
We evaluate a strongly regularised version of the Hastings-Levitov model HL$(\alpha)$ for $0\leq \alpha<2$. Previous results have concentrated on the small-particle limit where the size of the attaching particle approaches zero in the limit. However, we consider the case where we rescale the whole cluster by its capacity before taking limits, whilst keeping the particle size fixed. We first consider the case where $\alpha=0$ and show that under capacity rescaling, the limiting structure of the cluster is not a disk, unlike in the small-particle limit. Then we consider the case where $0<\alpha<2$ and show that under the same rescaling the cluster approaches a disk. We also evaluate the fluctuations and show that, when represented as a holomorphic function, they behave like a Gaussian field dependent on $\alpha$. Furthermore, this field becomes degenerate as $\alpha$ approaches 0 and 2, suggesting the existence of phase transitions at these values.
\end{abstract}
\tableofcontents
\section{Introduction}
Random growth occurs in many real world settings, for example we see it exhibited in the growth of tumours and bacterial growth. We would like to be able to model such processes to determine their behaviour in their scaling limits. Since the the 1960's, models have been built in order to describe individual processes. Perhaps the most famous examples of such models are the Eden model \cite{eden} and Diffusion Limited Aggregation (DLA) \cite{dla}. The Eden model is used to describe bacterial colony growth, whereas, DLA describes mineral aggregation (see for example \cite{rohde2005some}).
\\
\\
In their 1998 paper \cite{hastings1998laplacian},  Hastings and Levitov introduced a one parameter family of conformal maps HL$(\alpha)$ which can be used to model Laplacian growth processes and allows us to vary between the previous models by varying the parameter $\alpha$. In contrast to many well studied lattice based models, HL$(\alpha)$ is formed by using conformal mappings \cite{hastings1998laplacian}. We can then use complex analysis techniques to evaluate the growth. We consider a regularised version of this model and show that at certain values of $\alpha$ a phase transition on the scaling limits occurs. 
\subsection{Outline of the model}

In order to define our model we start by  defining the single particle map. Define $\Delta$ as the exterior of the unit disk in the complex plane, $\Delta=\{ |z|>1\}$. For any conformal map $f:\Delta\to \mathbb{C}$ we define the logarithmic capacity of the map to be, $$\lim_{z\to \infty}\log\left(f'(z)\right):=\log f'(\infty).$$ 
For each $c>0$, we then choose a general single particle mapping $f_{c}:\Delta\to \mathbb{C} \backslash K$ which takes the exterior of the unit disk to itself minus a particle of logarithmic capacity $c>0$ at $z=1$. Note that we can then rescale and rotate the mapping $f_{c}(z)$ to allow any attaching point on the boundary of the unit disk by letting $f_n(z)=e^{i\theta_n} f_{c_n}(ze^{-i \theta_n})$ where $\theta_n$ is the attaching angle and $c_n$ is the logarithmic capacity of the $n^{\text{th}}$ particle map $f_{c_n}(z)$.  
\\
\\
We can now form the cluster by composing the single particle maps. Let $K_0= \Delta^c=\{|z|\leq 1\}.$
 Suppose that we have some compact set $K_n$ made up of $n$ particles. We can find a bi-holomorphic map which fixes $\infty$ and takes the exterior of the unit disk to the complement of $K_n$ in the complex plane, $\phi_n: \Delta \to \mathbb{C}\backslash K_n.$
We then define the map $\phi_{n+1}$ inductively;
$$\phi_{n+1}=\phi_n \circ f_{n+1}=f_{1}\circ f_{2} \circ .... \circ f_{n+1}.$$
\\
There are several possible choices for the particle map $f_{c}$. The choice we make is  determined by what shape we would like the attaching particle to have. Hastings and Levitov introduce both the strike and bump mappings in \cite{hastings1998laplacian}. The strike map attaches a single slit onto the boundary at $z=1$ whereas the bump map attaches a particle with non-empty interior. We would like results to not be dependent of the specific choice of particle shape. In \cite{norris2019scaling}, Norris et al show there exists some absolute constant $c_0$ such that, provided $0\leq c <c_0$, all of the aforementioned shapes of particles, and indeed many other natural choices, satisfy the condition,
\begin{equation}\label{eqn}
f_{c}(z)=e^{c}z\exp\left(\frac{2c}{z-1}+\delta_c \left(z\right)\right)
\end{equation}
where $\delta_c(z)$ is some function of $z$ with $|\delta_c(z)|<\frac{\tilde{\lambda }c^{\frac{3}{2}}|z|}{|z-1|(|z|-1)}$ and $\tilde{\lambda}\in [0, \infty)$ is some constant not dependent on $c$ or $z$. Therefore, we take our single particle mappings from a class of particles satisfying (\ref{eqn}) for fixed $\tilde{\lambda}$. In the proofs that follow it will become clear that our results do not depend on the precise value of $\delta_c(z)$.
\\
\\
Now it just remains to define how the attaching points $\theta_n$ and capacities $c_n$ are chosen. We want to model Laplacian growth and so we choose the $\theta_n$ to be uniformly distributed, independent for each $n$, on the circle. This choice is made because after renormalisation of $\phi_n$, the Lebesgue measure of the unit circle under the image of $\phi_n$ is harmonic measure as seen from infinity \cite{rohde2005some}, and the harmonic measure of a portion of the unit circle is just the arclength of that portion rescaled by $2 \pi$. 
\\
\\
Finally, we must choose how the capacities $c_n$ are  distributed. Hastings and Levitov \cite{hastings1998laplacian} introduced a parameter $\alpha$ in order to distinguish between the various individual models they would like to encode within this one model for Laplacian growth. They choose, $$c_n=c|\phi_{n-1} '(e^{i\theta_n})|^{-\alpha}$$ for some $c>0$. This gives an off-lattice version of the Eden model when $\alpha=1$ and DLA when $\alpha=2$. In Section 3, we show that the total capacity, $\phi_n'(\infty)$ is well approximated by $(1+\alpha c n)^{\frac{1}{\alpha}}$. Therefore, if we define a version of HL$(\alpha)$ using the very strong regularisation $\tilde{c}_n=c|\phi_{n-1} '(\infty)|^{-\alpha}$, $\tilde{c}_n$ is approximately given by
\begin{equation} \label{eqn3}
c_n^*:=\frac{c}{1+\alpha c(n-1)}.
\end{equation}
In what follows, we denote $\phi_n=f_1\circ...\circ f_n$ where $f_n(z)=e^{i\theta_n} f_{c_n^{*}}(ze^{-i \theta_n})$ with $\theta_n$ i.i.d uniform on $[0,2\pi]$. We then keep $c$ fixed and rescale the cluster by its total capacity and evaluate the shape of the rescaled cluster $e^{-\sum_{i=1}^n c_{i}^{*}}\phi_n$ as $n\to \infty$.
\subsection{Previous work}
With the model now defined we can outline the work already done in this area. Most work has been done in the small-particle limit. This method involves evaluating the limiting cluster $\phi_n$ as we send the particle capacity $c\to 0$ while sending $n\to \infty$ with $nc\sim t$ for some $t$. Using this method Turner and Norris show that for $\alpha=0$ the limiting cluster in the small particle case behaves like a growing disk \cite{norris2012hastings}. Furthermore, Turner, Viklund and Sola show that in the small particle limit the shape of the cluster in a regularised setting approaches a circle for all $\alpha\geq 0$ provided the regularisation is sufficient \cite{viklund2015small}. Moreover, Silvestri \cite{silvestri2017fluctuation} shows that the fluctuations on the boundary, for HL$(0)$, in this small particle limit can be characterised by a log-correlated Gaussian field.
\\
\\
A different approach to that of the small-particle limit is to  not let $c\to0$ as $n\to\infty$, but instead, the limit of the cluster is found by rescaling the whole cluster by the capacity of the cluster at time $n$, before taking limits as the number of particles tends to infinity.  Rhode and Zinsmeister introduce a regularisation to the Hastings-Levitov model and show that in the case of $\alpha=0$ the limiting cluster under capacity rescaling exists and has finite length \cite{rohde2005some}.
\\
\\
Our work will follow the second approach. We will use results and ideas from the papers listed above, and in particular methods from \cite{norris2019scaling}, in order to characterise the limiting shape of the cluster in a regularised setting for $0\leq \alpha <2$ and then evaluate the fluctuations. Our results break down for $\alpha\geq 2$, this will be the subject of future work.

\subsection{Statement of results}
We first consider the case where $\alpha=0$ and show that under capacity rescaling, the limiting structure of the rescaled cluster is not a disk. This comes in the form of the following theorem. 
\begin{theorem}
Given any sequence $\{\theta_k\}_{1\leq k\leq n}$ of angles between $0$ and $2\pi$ and $c>0$, set $\Psi_n=f_1\circ...\circ f_n$ where $f_k(z)=e^{i\theta_k}f_c(e^{-i\theta_k} z)$ and $f_c(z)$ is any fixed capacity map in the class of particles given by (\ref{eqn}). There exists some $c_0>0$, which depends only on $\tilde{\lambda}$ such that for all $0<c<c_0$, there exists an $\epsilon>0$ such that for all $r>1$,
$$\limsup_{n\to \infty} \sup_{|z|>r}|e^{-cn}\Psi_{n}(z)-z|>\epsilon . $$
In particular if  $\{\theta_k\}_{1\leq k\leq n}$ are i.i.d uniform on $[0,2 \pi]$ then $\Psi_n$ is the HL$(0)$ process and the statement above shows that HL$(0)$ does not converge to a disk under capacity rescaling.
\end{theorem}
This result is particularly interesting because it is independent of our choice of angles. If we have a constant capacity map of the right form then there is no possible way to choose the angles so that under capacity rescaling the limiting cluster looks like a disk.
\\
\\
Next we consider the case where $0<\alpha<2$ and show that under capacity rescaling the $HL(\alpha)$ cluster approaches a disk. We then evaluate the fluctuations and show that they behave like a Gaussian field dependent on $\alpha$. Our two main results are stated as follows.
\begin{theorem}\label{a}
For $0<\alpha<2$, let the map $\phi_n$ be defined as above with $c_n^{*}$ as defined in (\ref{eqn3}) and $\theta_n$ i.i.d uniform on $[0,2 \pi]$. Then for any $r>1$,
$$\mathbb{P}\left(\limsup_{n\to \infty}\left\lbrace\sup_{|z|\geq r}|e^{-\sum_{i=1}^nc_i^*}\phi_n(z)-z|>\frac{\log n}{\sqrt{n}}\right\rbrace\right)=0.$$
\end{theorem}
This result tells us we have uniform convergence of our cluster in the exterior disk to a disk. The following result shows that the fluctuations behave like a Gaussian field. 
\begin{theorem} \label{6.10}
Let $0<\alpha<2$ and $\phi_n$ be defined as in Theorem \ref{a}. Then as $n\to \infty$,
$$\sqrt{n}\left(e^{-\sum_{i=1}^{n}c_i^{*}}\phi_{n}(z)-z\right)\to \mathcal{F}(z)$$
in distribution on $\mathcal{H}$, where $\mathcal{H}$ is the space of holomorphic functions on ${|z| > 1}$,  equipped with a suitable metric $\textbf{d}_{\mathcal{H}}$ defined later, and where 
$$\mathcal{F}(z)=\sum_{m=0}^{\infty}(A_m+iB_m)z^{-m}$$
with $A_m$, $B_m \sim  \mathcal{N}\left(0, \frac{2}{\alpha(2m+2-\alpha)}\right)$ and $A_m$, $B_k$ independent for all choices of $m$ and $k$.
\end{theorem}
Notice that it is clear this result does not hold for $\alpha=0$ or $\alpha=2$. This is in contrast to \cite{viklund2015small} where results hold for all $\alpha\geq 0$ and suggests a phase transition at these values.

\subsection{Outline of the paper}
The outline of the paper is as follows. In Section 2 we will show that for clusters formed by composing maps of constant capacity and of a certain form, we can not pick a sequence of angles so that the limiting cluster under capacity rescaling approaches a disk. In particular, under capacity rescaling HL$(0)$ is not a growing disk. Then in Section 3 we will show that our choice of capacities is a good approximation to the regularisation of HL$(\alpha)$ at $\infty$.  In Section 4, we show that the pointwise limit of the cluster for $0<\alpha<2$ is a disk and then in Section 5 we will use a Borel-Cantelli argument to show we have uniform convergence on the exterior disk. Finally, in Section 6 we will evaluate the fluctuations for $0<\alpha<2$ and show that they are distributed according to a Gaussian field dependent on $\alpha$.
\section{The case where $\alpha=0$}
We want to evaluate the limiting shape of our random cluster. We first deal with the case where $\alpha=0$. We will show in this section that in the limit HL$(0)$ does not approach a disk. Furthermore, we will prove a stronger statement that for clusters formed by composing maps of constant capacity, in the class of particles defined in (\ref{eqn}), we can not approach a disk under capacity rescaling. We note that in the case where $\alpha=0$ our regularisation does not effect the model, so this result holds for HL$(0)$ under no regularisation. Our proof is reliant on the fact that under capacity rescaling the limit for HL$(0)$ exists, this result was proved by Rhode and Zinsmeister in \cite{rohde2005some}. One might expect, given this result, that the scaling limit is a growing disk, this would agree with the result in the small particle limit \cite{norris2012hastings}. However, the following theorem proves this does not hold.
\\
\begin{thm1}
Given any sequence $\{\theta_k\}_{1\leq k\leq n}$ of angles between $0$ and $2\pi$ and $c>0$, set $\Psi_n=f_1\circ...\circ f_n$ where $f_k(z)=e^{i\theta_k}f_c(e^{-i\theta_k} z)$ and $f_c(z)$ is any fixed capacity map in the class of particles given by (\ref{eqn}). There exists some $c_0>0$, which depends only on $\tilde{\lambda}$ such that for all $0<c<c_0$, there exists an $\epsilon>0$ such that for all $r>1$,
$$\limsup_{n\to \infty} \sup_{|z|>r}|e^{-cn}\Psi_{n}(z)-z|>\epsilon . $$
In particular if  $\{\theta_k\}_{1\leq k\leq n}$ are i.i.d uniform on $[0,2 \pi]$ then $\Psi_n$ is the HL$(0)$ process and the statement above shows that HL$(0)$ does not converge to a disk under capacity rescaling.
\end{thm1}

\begin{proof}
First suppose this does not hold. Then for any $\epsilon>0$,
$$\limsup_{n\to \infty} \sup_{|z|>r}|e^{-cn}\Psi_{n}(z)-z|<\epsilon.$$
Then we can write,
$$ |e^{-cn}\Psi_{n}(z)-z|=\left|\left(e^{-c}f_{n}(z)-z\right)+e^{-c}\left(e^{-c(n-1)}\Psi_{n-1}(f_{n}(z))-f_{n}(z)\right)\right|$$
which we can bound below for all $|z|>r$ as follows,
$$ |e^{-cn}\Psi_{n}(z)-z|\geq |e^{-c}f_{n}(z)-z|- \sup_{|z|>r}|e^{-c}||e^{-c(n-1)}\Psi_{n-1}(f_{n}(z))-f_{n}(z)|.$$
We can then take the supremum of both sides, and use that $|f_{n}(z)|>r$ for all $|z|>r$, to reach the following bound on the supremum,
\begin{equation}\label{bound}
\sup_{|z|>r}|e^{-cn}\Psi_{n}(z)-z|\geq \sup_{|z|>r}|e^{-c}f_{n}(z)-z|- \sup_{|z|>r}|e^{-c}||e^{-c(n-1)}\Psi_{n-1}(z)-z|.
\end{equation}
So consider $\sup_{|z|>r}|e^{-c}f_{n}(z)-z|$. Using the definition of $f_{c}(z)=e^{c}z\exp\left(\frac{2c}{z-1}+\delta_c \left(z\right)\right)$ we can bound this below by, 
$$\sup_{|z|>r}|e^{-c}f_{n}(z)-z|\geq \sup_{|z|>r}|z|\left|\exp\left(\frac{2c}{z-1}+\delta_c \left(z\right)\right)-1\right|.$$
Then by using the integral form of Taylor's remainder formula we see that for any complex $x$, $|e^x-(1+x)|\leq |x|^2 e^{|x|}$ and therefore,
$$|e^x-1|\geq |x|-|x|^2 e^{|x|}.$$
Hence, we can find a lower bound on the expression above,
\begin{align*}
\sup_{|z|>r}|e^{-c}f_{n}(z)-z|
&\geq \sup_{|z|>r}|z|\left(\left|\frac{2c}{z-1}+\delta_c \left(z\right)\right|-\left|\frac{2c}{z-1}+\delta_c \left(z\right)\right|^2\exp\left(\left|\frac{2c}{z-1}+\delta_c \left(z\right)\right|\right)\right)\\
&\geq \lim_{z\to \infty}|z|\left(\left|\frac{2c}{z-1}+\delta_c \left(z\right)\right|-\left|\frac{2c}{z-1}+\delta_c \left(z\right)\right|^2\exp\left(\left|\frac{2c}{z-1}+\delta_c \left(z\right)\right|\right)\right)\\
&\geq 2c-\tilde{\lambda}c^{\frac{3}{2}}
\end{align*}
where $\tilde{\lambda}$ is defined as in (\ref{eqn}). Therefore, combining this inequality with the inequality (\ref{bound}) gives,
$$\sup_{|z|>r}|e^{-cn}\Psi_{n}(z)-z|\geq 2c-\tilde{\lambda}c^{\frac{3}{2}}- \sup_{|z|>r}|e^{-c}||e^{-c(n-1)}\Psi_{n-1}(z)-z|.$$
Taking the limit supremum and using our initial assumption we have,
\begin{align*}
\limsup_{n\to \infty} \sup_{|z|>r}|e^{-cn}\Psi_{n}(z)-z|&\geq 2c-\tilde{\lambda}c^{\frac{3}{2}}-e^{-c}\epsilon.
\intertext{ So choose $0\leq c\leq \frac{4}{\tilde{\lambda}^2}$, then $2c-\tilde{\lambda}c^{\frac{3}{2}}>0$ and if we choose $\epsilon=\frac{2c-\tilde{\lambda}c^{\frac{3}{2}}}{(1+e^{-c})}>0$ then}
\limsup_{n\to \infty} \sup_{|z|>r}|e^{-cn}\Psi_{n}(z)-z|&\geq \epsilon
\end{align*}
a contradiction. 
\end{proof}
This is a strong result because it proves that if we have a cluster which is composed of functions of the right form, no matter how we pick our sequence of attaching angles $\{\theta_{n} \}$ the limiting structure of the cluster, when rescaled by its capacity, does not approach a disk. 
\section{Regularisation}
The aim of this section is to provide some justification for the choice of $c_n^{*}$ as an approximation to the regularisation of HL$(\alpha)$ at $\infty$. Recall that we choose,
$$c_n^*=\frac{c}{1+\alpha c(n-1)}.$$ We start by providing some notation used throughout the remainder of the paper. 
Let $\phi_k$ and $c_i^*$ be defined as above, then  we denote $$C^{*}_{k,n}=\sum_{i=k}^n c_i^*.$$

\subsection{Error term evaluation}
In order to more easily apply complex analysis methods to our cluster we would like to write the sum $C_{1,n}^{*}$ in a simplified form. We do so by providing the following approximation on the sum, subject to an error term which converges to $0$, uniformly in $k$, as $n\to \infty$.
\begin{lemma}\label{lemma2.2}
For $c_n^*=\frac{c}{1+\alpha c(n-1)}$ we have the following equality;
$$C_{k,n}^{*}=\frac{1}{\alpha}\log\left(\frac{1+\alpha c n}{1+\alpha c(k-1)}\right)(1+\epsilon_{k,n})$$ where $$0<\epsilon_{k,n}<\frac{\alpha^2 c^2 (n-k+1)}{(1+\alpha c(k-1))(1+\alpha c n)\log\left(\frac{1+\alpha c n}{1+\alpha c(k-1)}\right)}\leq \frac{\alpha c}{\log(1+\alpha c n)}.$$
Therefore, $\epsilon_{k,n}\to 0$, uniformly in $k$, as $n\to \infty$. 
\end{lemma}
\begin{proof}
We will approximate the sum with 
\begin{align*}
\frac{1}{\alpha}\log\left(\frac{1+\alpha c n}{1+\alpha c(k-1)}\right)=\int_k^{n+1} \frac{c}{1+\alpha c(x-1)} dx .
\end{align*}
Then
\begin{align*}
C_{k,n}^{*}-\frac{1}{\alpha}\log\left(\frac{1+\alpha c n}{1+\alpha c(k-1)}\right)&=\sum_{i=k}^n \left( c_i^{*}-\int_i^{i+1} \frac{c}{1+\alpha c(x-1)} dx \right)\\
&\leq \sum_{i=k}^n \left(c_i^* -c_{i+1}^* \right)\\
&=\frac{\alpha c^2 (n-k+1)}{(1+\alpha c(k-1))(1+\alpha c n)}.
\intertext{Thus,} 0<\epsilon_{k,n} &<\frac{\alpha^2 c^2 (n-k+1)}{(1+\alpha c(k-1))(1+\alpha c n)\log\left(\frac{1+\alpha c n}{1+\alpha c(k-1)}\right)}.
\end{align*}
So we consider,
\begin{align*}
&\sup_{k\leq n}\frac{\alpha^2 c^2 (n-k+1)}{(1+\alpha c(k-1))(1+\alpha c n)\log\left(\frac{1+\alpha c n}{1+\alpha c(k-1)}\right)}\\
&=\frac{\alpha^2 c^2}{1+\alpha c n}\sup_{k\leq n}\frac{n-k+1}{(1+\alpha c(k-1))\log\left(\frac{1+\alpha c n}{1+\alpha c(k-1)}\right)}.
\end{align*}
So let us find, 
$$\sup_{k\leq n}\frac{n-k+1}{(1+\alpha c(k-1))\log\left(\frac{1+\alpha c n}{1+\alpha c(k-1)}\right)}.$$
Let $x=1+\alpha c (k-1)$ and find the derivative
\begin{align*}
\frac{d}{dx}\left(\frac{1+\alpha c n -x}{x \log\left(\frac{1+\alpha c n }{x}\right)}\right)
=\frac{(1+\alpha c n )-(1+\alpha c n) \log\left(\frac{1+\alpha c n }{x}\right) -x }{x^2 \left(\log\left(\frac{1+\alpha c n }{x}\right)\right)^2}.
\end{align*}
The numerator in this fraction is increasing and from this it is clear that the derivative is negative. Therefore the maximum occurs when $k=1$. Thus,
$$0\leq \epsilon_{k,n}\leq \frac{\alpha^2 c^2}{1+\alpha c n}\frac{n}{\log\left(1+\alpha c n\right)}\leq \frac{\alpha c}{\log(1+\alpha c n)}.$$
Furthermore, taking the limit as $n\to \infty$ we have $\epsilon_{k,n}\to 0$, uniformly in $k$, as claimed.
\end{proof}
The following corollary provides a nice bound on $(1+\alpha c k)^{1+\epsilon_{k,n}}$ which will make computations in later sections easier. 
\begin{corollary}\label{3.3}
Let $\epsilon_{k,n}$ be defined as in Lemma \ref{lemma2.2}. Then for $1\leq k\leq n$ and $\alpha \geq 0$ the following bound holds,
$$(1+\alpha c k)^{1+\epsilon_{k,n}}\leq (1+\alpha ce^{\alpha c})(1+\alpha c k).$$
\end{corollary}
\begin{proof}
We can write 
$$(1+\alpha c k)^{1+\epsilon_{k,n}}=(1+\alpha c k)(1+\alpha c k)^{\epsilon_{k,n}}=(1+\alpha c k)(1+(1+\alpha c k)^{\epsilon_{k,n}}-1).$$
So let $\delta_{k,n}=(1+\alpha c k)^{\epsilon_{k,n}}-1$, then
$$\delta_{k,n}= (e^{\epsilon_{k,n}\log(1+\alpha c k)}-1)\leq \epsilon_{k,n}\log(1+\alpha c k)e^{\epsilon_{k,n}\log(1+\alpha c k)}.$$
We have just shown that 
$$|\epsilon_{k,n}|\leq \frac{\alpha c}{\log\left(1+\alpha c n\right)}.$$
So, $$0\leq |\delta_{k,n}|\leq \alpha ce^{\alpha c}.$$
Therefore, 
$$(1+\alpha c k)^{1+\epsilon_{k,n}}\leq (1+\alpha c k)(1+\alpha ce^{\alpha c}).$$
\end{proof}
\subsection{Regularisation approximation}
With the estimates provided above we can now provide justification for our choice of $c_n^{*}$. We start by providing some more notation. For each $n\in \mathbb{N}$,  $c$ defined as above we  denote $\phi^{\infty}_n=\phi^{\infty}_{n-1}\circ f^{\infty}_n$ where $f^{\infty}_n(z)=e^{i\theta_n} f_{\tilde{c}_n}(ze^{-i \theta_n})$ with $\theta_n$ i.i.d uniform on $[0,2\pi]$ and
$$\tilde{c}_n=\frac{c}{\left|\left(\phi^{\infty}_{n-1}\right)'(\infty)\right|^{\alpha}}.$$ Furthermore, we define,
$$\tilde{C}_{k,n}=\sum_{i=k}^{n} \tilde{c}_i.$$
The maps $\phi^{\infty}_n$ correspond to the true model for HL$(\alpha)$ regularised at $\infty$. The aim of the remainder of this section will be to prove the following theorem.
\begin{proposition}\label{reg}
For $C_{1,n}^{*}$ and $\tilde{C}_{1,n}$ defined as above, the following inequality holds,
\begin{align*}
0\leq C_{1,n}^{*}-\tilde{C}_{1,n}\leq 6 c
\end{align*}
Furthermore, 
\begin{align*}
\tilde{c}_n=c_n^{*}(1+\epsilon_n^{\infty})
\end{align*}
where $\epsilon_n^{\infty}\to 0$, uniformly in $n$, as $c\to 0$.
\end{proposition}
Therefore if we choose our $c$ sufficiently small we see that our regularisation is a good approximation to regularisation at infinity. In order to prove Proposition \ref{reg} we first form a difference equation on $C_{1,n}^{*}$. 
\begin{lemma}\label{kap}
With $C_{1,n}^{*}$ defined as above the following equality holds
$$C_{1,n}^{*}=C_{1,n-1}^{*}+ce^{-\alpha C_{1,n-1}^{*}}+\kappa_n$$
where $0\leq \kappa_n\leq \frac{2\alpha c^2}{1+\alpha c(n-1)}$.
\end{lemma}
\begin{proof}
Let 
\begin{align*}
\kappa_n &=\left(C_{1,n}^{*}-C_{1,n-1}^{*}\right)-ce^{-\alpha C_{1,n-1}^{*}}.\intertext{Then by the definition of $C_{1,n}^{*}$,}\kappa_n&=c_n^{*}-ce^{-\alpha C_{1,n-1}^{*}}. \intertext{Thus, using the approximation from Lemma \ref{lemma2.2},}
\kappa_n&=c_n^{*}-\frac{c}{(1+\alpha c(n-1))^{1+\epsilon_{1,n-1}}}\\
&=\frac{c}{1+\alpha c(n-1)}\left(1-\frac{1}{(1+\alpha c(n-1))^{\epsilon_{1,n-1}}}\right)\\
&=\frac{c}{1+\alpha c(n-1)}\left(1-\exp\left(-\epsilon_{1,n-1}\log(1+\alpha c(n-1))\right)\right).
\intertext{Since $\epsilon_{1,n-1}$ is small for small $c$ we can Taylor expand the exponential to get,}
\kappa_n&=\frac{c}{1+\alpha c(n-1)}\left(\epsilon_{1,n-1}\log(1+\alpha c(n-1))-r(n,c)\right).
\end{align*}
where $r(n,c)$ is the remainder term in the Taylor expansion. From Lemma \ref{lemma2.2} we know $0\leq \epsilon_{1,n-1}\leq \frac{\alpha c}{\log(1+\alpha c(n-1))}$. Moreover, $0\leq r(n,c)\leq e^{\alpha c}(\epsilon_{1,n-1}\log(1+\alpha c(n-1)))^2$, so for $c$ sufficiently small,
$$0\leq \kappa_n\leq \frac{2\alpha c^2}{1+\alpha c(n-1)}.$$
\end{proof}
We can now show that $C_{1,n}^{*}$ and $\tilde{C}_{1,n}$ are sufficiently close by proving Proposition \ref{reg}.
\begin{proof}[Proof of Proposition \ref{reg}]
We will prove the statement inductively. By definition, \newline $C_{1,1}^{*}-\tilde{C}_{1,1}=0$. So assume, \begin{align*}
0\leq C_{1,n-1}^{*}-\tilde{C}_{1,n-1}\leq 6 c
\end{align*}
Then note that since,
\begin{align*}
\tilde{c}_n&=\frac{c}{\left|\left(\phi^{\infty}_{n-1}\right)'(\infty)\right|^{\alpha}}
\intertext{then}
\tilde{c}_n&=\frac{c}{\left(e^{\tilde{C}_{1,n-1}}\right)^{\alpha}}.
\end{align*}
Furthermore,
\begin{align*}
\tilde{C}_{1,n}&=\tilde{C}_{1,n-1}+\tilde{c}_n. 
\intertext{Therefore,}
\tilde{C}_{1,n}&=\tilde{C}_{1,n-1}+\frac{c}{\left(e^{\tilde{C}_{1,n-1}}\right)^{\alpha}}.
\end{align*}
Thus, by Lemma \ref{kap},
\begin{align*}
C_{1,n}^{*}-\tilde{C}_{1,n}&=\left(C_{1,n-1}^{*}-\tilde{C}_{1,n-1}\right)+c\left(e^{-\alpha C_{1,n-1}^{*}}-e^{-\alpha \tilde{C}_{1,n-1}}\right)+\kappa_n \\
&=\left(C_{1,n-1}^{*}-\tilde{C}_{1,n-1}\right)+ce^{-\alpha C_{1,n-1}^{*}}\left(1-e^{\alpha\left(C_{1,n-1}^{*}- \tilde{C}_{1,n-1}\right)}\right)+\kappa_n.
\intertext{Taylor expanding the $e^{\alpha\left(C_{1,n-1}^{*}- \tilde{C}_{1,n-1}\right)}$ term gives,}
C_{1,n}^{*}-\tilde{C}_{1,n}&=\left(C_{1,n-1}^{*}-\tilde{C}_{1,n-1}\right)+c\alpha e^{-\alpha C_{1,n-1}^{*}}\left(\tilde{C}_{1,n-1}-C_{1,n-1}^{*}-r(n,c)\right)+\kappa_n\\
&=\left(C_{1,n-1}^{*}-\tilde{C}_{1,n-1}\right)\left(1-c\alpha e^{-\alpha C_{1,n-1}^{*}}\right)+\left(\kappa_n- r(n,c)c\alpha e^{-\alpha C_{1,n-1}^{*}}\right).
\end{align*}
where $r(n,c)$ is the Taylor remainder term. We know $r(n,c)=\frac{e^{\xi}}{2}\alpha \left(C_{1,n-1}^{*}-\tilde{C}_{1,n-1}\right)^2$ for some $0<\xi<\alpha \left(C_{1,n-1}^{*}-\tilde{C}_{1,n-1}\right)$. Thus, under our assumption that \newline $0\leq C_{1,n-1}^{*}-\tilde{C}_{1,n-1}\leq 6c$, we have, 
\begin{align*}
0&\leq r(n,c)c\alpha e^{-\alpha C_{1,n-1}^{*}}\leq \frac{c^3 \alpha^2 e^{6 \alpha c}}{1+\alpha c(n-1)}.
\intertext{Then if $c$ is small enough}
0&\leq r(n,c)c\alpha e^{-\alpha C_{1,n-1}^{*}}\leq \frac{\kappa_n}{2}.
\end{align*}
Therefore, let $\tilde{\kappa}_n=\kappa_n- r(n,c)c\alpha e^{-\alpha C_{1,n-1}^{*}}$, then $0\leq \tilde{\kappa}_n\leq \frac{2\alpha c^2}{1+\alpha c(n-1)}$. Hence,
\begin{align*}
C_{1,n}^{*}-\tilde{C}_{1,n}&=\left(C_{1,n-1}^{*}-\tilde{C}_{1,n-1}\right)\rho_{n-1}+\tilde{\kappa}_n
\intertext{where $\rho_{n-1}=1-c\alpha e^{-\alpha C_{1,n-1}^{*}}$. So,}
C_{1,n}^{*}-\tilde{C}_{1,n}&=\left(C_{1,n-2}^{*}-\tilde{C}_{1,n-2}\right)\rho_{n-2}\rho_{n-1}+\tilde{\kappa}_{n-1}\rho_{n-1}+\tilde{\kappa}_n\\
&=\left(C_{1,1}^{*}-\tilde{C}_{1,1}\right)\prod_{i=1}^{n-1} \rho_i+\sum_{j=2}^{n-1}\left(\tilde{\kappa}_j\prod_{k=j}^{n-1} \rho_k\right)+\tilde{\kappa}_n
\intertext{but since $\left(C_{1,1}^{*}-\tilde{C}_{1,1}\right)=0$,}
C_{1,n}^{*}-\tilde{C}_{1,n}&=\sum_{j=2}^{n-1}\left(\tilde{\kappa}_j\prod_{k=j}^{n-1} \rho_k\right)+\tilde{\kappa}_n.
\end{align*}
We first analyse $\prod_{k=j}^{n-1} \rho_k$,
\begin{align*}
\prod_{k=j}^{n-1} \rho_k&=\prod_{k=j}^{n-1}\left(1-c\alpha e^{-\alpha C_{1,k-1}^{*}}\right)\\
&=\prod_{k=j}^{n-1}\left(1-\frac{\alpha c}{(1+\alpha c(k-1))^{1+\epsilon_{1,k-1}}}\right)\\
&=\exp\left(\sum_{k=j}^{n-1}\log\left(1-\frac{\alpha c}{(1+\alpha c(k-1))^{1+\epsilon_{1,k-1}}}\right)\right).
\intertext{Using the Taylor expansion of $\log\left(1-\frac{\alpha c}{(1+\alpha c(k-1))^{1+\epsilon_{1,k-1}}}\right)$ we have,}
\prod_{k=j}^{n-1} \rho_k&=\exp\left(\sum_{k=j}^{n-1}\frac{-\alpha c}{(1+\alpha c(k-1))^{1+\epsilon_{1,k-1}}}\right)\exp\left(\sum_{k=j}^{n-1}\tilde{r}(k,c)\right)
\intertext{where $\tilde{r}(k,c)$ is the Taylor remainder term. But since for each $2\leq k\leq n-1$, \newline $\sum_{k=j}^{n-1}\tilde{r}(j,c)\leq 0$ in the expansion of $\log\left(1-\frac{\alpha c}{(1+\alpha c(k-1))^{1+\epsilon_{1,k-1}}}\right)$,}
0\leq \prod_{k=j}^{n-1} \rho_k&\leq \exp\left(\sum_{k=j}^{n-1}\frac{-\alpha c}{(1+\alpha c(k-1))^{1+\epsilon_{1,k-1}}}\right).
\end{align*}
By Corollary \ref{3.3},
\begin{align*}
0\leq (1+\alpha c (k-1))^{1+\epsilon_{1,k-1}}&\leq (1+\alpha ce^{\alpha c})(1+\alpha c (k-1)).
\intertext{Therefore,}
\frac{-\alpha c}{(1+\alpha c(k-1))^{1+\epsilon_{1,k-1}}}&\leq \frac{-\alpha c}{(1+\alpha ce^{\alpha c})(1+\alpha c(k-1))}
\intertext{Thus,}
0\leq \prod_{k=j}^{n-1} \rho_k&\leq e^{\frac{1}{(1+\alpha ce^{\alpha c})}}\frac{1+\alpha c(j-1)}{1+\alpha c (n-2)}.
\end{align*}
Finally we see that,
\begin{align*}
\sum_{j=2}^{n-1}\left(\tilde{\kappa}_j\prod_{k=j}^{n-1} \rho_k\right)&\leq 6 \alpha c^2 \sum_{j=2}^{n-1}\frac{1}{1+\alpha c(j-1)}\frac{1+\alpha c(j-1)}{1+\alpha c (n-2)}\\
&\leq 6 \alpha c^2 \frac{n-3}{1+\alpha c(n-2)}.\\
\intertext{Thus for $c$ small enough,}
0\leq \sum_{j=2}^{n-1}\left(\tilde{\kappa}_j\prod_{k=j}^{n-1} \rho_k\right)&\leq 6c.
\end{align*}
Hence, since for all $n$, and $c$ sufficiently small, $$0\leq \tilde{\kappa}_n\leq 6 c$$ it follows that,
\begin{equation}\label{starbound}
0\leq C_{1,n}^{*}-\tilde{C}_{1,n}\leq 6 c.
\end{equation}
Finally, consider $\frac{\tilde{c}_n}{c_n^{*}}$,
\begin{align*}
\frac{\tilde{c}_n}{c_n^{*}}&=\frac{(1+\alpha c(n-1))}{e^{\alpha \tilde{C}_{1,n}}}\\
&=e^{\alpha\left(\frac{1}{\alpha}\log(1+\alpha c(n-1))-C^{*}_{1,n}\right)}e^{\alpha\left(C^{*}_{1,n}-\tilde{C}_{1,n}\right)}
\intertext{Thus by (\ref{starbound}),}
\frac{\tilde{c}_n}{c_n^{*}}&\leq e^{\alpha\left(\frac{1}{\alpha}\log(1+\alpha c(n-1))-C^{*}_{1,n-1}\right)}e^{6\alpha c}.
\intertext{Therefore using the bound in Lemma \ref{lemma2.2},}
\frac{\tilde{c}_n}{c_n^{*}}&\leq e^{\alpha\left(\epsilon_{1,n-1}\log(1+\alpha c(n-1))\right)}e^{6 \alpha c}\\
&\leq e^{ \alpha(\alpha+6) c}.
\end{align*}
Thus, $$\tilde{c}_n=c_n^{*}(1+\epsilon_n^{\infty})$$
where $\epsilon_n^{\infty}\to 0$ uniformly as $c\to 0$.
\end{proof}
Therefore, it follows that we can choose $c$ sufficiently small so that $\phi_n$ provides a good approximation to $\phi_n^{\infty}$. For notational simplicity all subsequent results are proved for $\phi_n$, however, it is straightforward to verify that $c$ can be chosen sufficiently small such that analogous results hold for $\phi_n^{\infty}$.
\section{Pointwise convergence for $0<\alpha<2$}
\subsection{Estimates}
In this section we will provide estimates for several variables which we will then call on throughout the rest of the paper. Whilst this work is an essential part of the analysis, we advise that the reader may skip the proofs of this section if they are only interested the main results of the paper. 
\\
\\
We start by providing some notation used throughout the remainder of the paper. Let $\phi_k$ and $c_i^*$ be defined as above. Recall, we denote $C^{*}_{k,n}=\sum_{i=k}^n c_i^*$. Then for any $z\in \mathbb{C}$ we define our increments $X_{k,n}(z)$ as;
\begin{equation}
X_{k,n}(z):=e^{-C_{1,n}^{*}}\left(\phi_k\left(e^{-C_{k+1,n}^{*}}z\right)-\phi_{k-1}\left(e^{-C_{k,n}^{*}}z\right)\right).
\end{equation}
Let $\mathcal{F}_{k-1}$ be the $\sigma$-algebra generated by the set $\{\theta_i \; : 1\leq i\leq k-1 \}$. We first show that for all $0<k \leq n$, $$\mathbb{E}(X_{k,n}(z) | \mathcal{F}_{k-1})=0.$$ This is shown in the following lemma.
\begin{lemma}\label{lemma4.2}
Define the sequence $\{X_{k,n}(z)\}_{k=0}^n$ and corresponding filtration $\mathcal{F}_k$ as above. For each $z\in \mathbb{C}$, the following property is satisfied for all $0<k \leq n$,
$$\mathbb{E}(X_{k,n}(z) | \mathcal{F}_{k-1})=0.$$
\end{lemma}

\begin{proof}
We first show;
\begin{align*}
\int_0^{2\pi} \phi_{k-1} (e^{i \theta} f_{c_k^*}(e^{-i \theta} z)) \frac{d \theta}{2 \pi}&=\phi_{k-1}(e^{c_k^*} z).
\intertext{Let $w=e^{i\theta}$, then the integral can be rewritten as}
\int_0^{2\pi} \phi_{k-1} (e^{i \theta} f_{c_k^*}(e^{-i \theta} z)) \frac{d \theta}{2 \pi}&=\frac{1}{2\pi i} \int_{C} \frac{\phi_{k-1}(w f_{c_k^*}(z/w))}{w} dw
\end{align*}
where $C$ is the unit circle centered at 0. The map $\phi_{k-1}(w f_{c_k^*}(z/w)):\Delta \to \Delta$  is analytic with a removable singularity at 0 and so by Cauchy's integral formula,
\begin{align*}
\frac{1}{2\pi i} \int_{C} \frac{\phi_{k-1}(w f_{c_k^*}(z/w))}{w} dw &=\lim_{w\to 0} \phi_{k-1}(wf_{c_k^*}(z/w))\\
&=\phi_{k-1}(\lim_{w\to 0} wf_{c_k^*}(z/w))\\
&=\phi_{k-1}\left(\lim_{w\to 0} \left(e^{c_k^*}z +a_0w+a_1 \frac{w^2}{z^2}+...\right)\right)
\end{align*}
for some complex number sequence of $a_i$'s. Thus,
$$\int_0^{2\pi} \phi_{k-1} (e^{i \theta} f_{c_k^*}(e^{-i \theta} z)) \frac{d \theta}{2 \pi}=\phi_{k-1}(e^{c_k^*} z)$$
as required. So now let us consider $\mathbb{E}(\phi_k(z)|\mathcal{F}_{k-1})$. This can be rewritten as 
 $$ \mathbb{E}(\phi_k(z)|\mathcal{F}_{k-1})=\mathbb{E}(\phi_{k-1} (e^{i \theta} f_{c_k^*}(e^{-i \theta} z))|\mathcal{F}_{k-1}).$$
The only randomness here comes from $\theta_k$, the $c_k^*$ are pre-determined, and so,
$$\mathbb{E}(\phi_k(z)|\mathcal{F}_{k-1})=\int_0^{2\pi} \phi_{k-1} (e^{i \theta} f_{c_k^*}(e^{-i \theta} z)) \frac{d \theta}{2 \pi}=\phi_{k-1}(e^{c_k^*} z).$$
Therefore,
$$\mathbb{E}(\phi_k(e^{C_{k+1,n}^{*}}z)|\mathcal{F}_{k-1})=\phi_{k-1}(e^{C_{k,n}^{*}}z).$$
Thus,
$$\mathbb{E}(X_{k,n}| \mathcal{F}_{k-1})=e^{-C_{1,n}^{*}}\left(\mathbb{E}(\phi_k(e^{C_{k+1,n}^{*}}z)|\mathcal{F}_{k-1})-\phi_{k-1}(e^{C_{k,n}^{*}}z)\right)=0$$
as required.

\end{proof}
Therefore, we also define the sum,
\begin{equation}\label{mart} M_n(z):=\sum_{k=1}^{n} X_{k,n}(z)=e^{-C_{1,n}^{*}}\phi_n(z)-z 
\end{equation}
which, by Lemma \ref{lemma4.2}, is a martingale difference array. We will also need to define the bounded variation
\begin{equation}\label{diff} 
T_n(z):=\sum_{k=1}^n\mathbb{E}(|X_{k,n}(z)|^2 \; |\; \mathcal{F}_{k-1}).
\end{equation}
Our aim is to show that we approach a disk pointwise, equivalently, for a fixed value $z$, $|M_n(z)|\to 0$ as $n\to \infty$. Throughout we use $\lambda$ to denote strictly positive, unless stated otherwise, constants which may change from line to line. Where these constants depend on  parameters from the model we indicate these explicitly. 
\\
\\
We will find pointwise bounds on $X_{k,n}(z)$ and $T_n(z)$. By definition; 
\begin{align*}
|X_{k,n}(z)|&=e^{-C_{1,n}^{*}}|\phi_k(e^{C_{k+1,n}^{*}} z)-\phi_{k-1}(e^{C_{k,n}^{*}} z)|\\
&=e^{-C_{1,n}^{*}}|\phi_{k-1}(e^{i \theta_k} f_{c_k^*}(e^{-i\theta_k}e^{C_{k+1,n}^{*}} z))-\phi_{k-1}(e^{C_{k,n}^{*}} z)|
\end{align*}
So we introduce the following parameterisation. 
\begin{definition} \label{etadef}
For each $n\in \mathbb{N}$, $z\in \mathbb{C}$, $k\leq n$ and $\delta_c(z)$ defined as in (\ref{eqn}), we define the following parameterisation for $0<s<1$,
$$\eta_{k,n}(s,z)=e^{C_{k,n}^{*}}z\exp\left(s\left(\frac{2c_k^{*}}{e^{-i\theta_k}e^{C_{k+1,n}^{*}}z-1}+\delta_c \left(e^{-i\theta_k}e^{C_{k+1,n}^{*}}z\right)\right)\right).$$
\end{definition}
We start by showing that for $|z|>r$, for some $r>1$, we can bound $\delta_c \left(e^{-i\theta_k}e^{C_{k+1,n}^{*}}z\right)$ by a constant via the following lemma.
\begin{lemma}\label{deltabound}
For $C_{k,n}^{*}$ and $\delta_c(z)$ defined as above, and for $|z|>r$ for some $r>1$, the following bound holds,
$$|\delta_{c_k^{*}} \left(e^{-i\theta_k}e^{C_{k+1,n}^{*}}z\right)|<\lambda(\alpha,c,r)k^{\frac{1}{\alpha}-\frac{3}{2}}n^{\frac{-1}{\alpha}}\leq \lambda(\alpha,c,r)k^{ -\frac{3}{2}}<\lambda(\alpha,c,r)$$
where $\lambda(\alpha,c,r)$ is a positive constant dependent on $\alpha$, $c$ and $r$.
\end{lemma}
\begin{proof}
From equation (\ref{eqn}) we know \begin{align*}
|\delta_{c_k^{*}}(z)|\leq &\frac{\tilde{\lambda }(c_k^*)^{\frac{3}{2}}|z|}{|z-1|(|z|-1)}  \intertext{where $\tilde{\lambda}$ is some constant. Therefore,}
\left|\delta_{c_k^{*}} \left(e^{-i\theta_k}e^{C_{k+1,n}^{*}}z\right)\right|&\leq \frac{\tilde{\lambda }(c_k^*)^{\frac{3}{2}}|e^{C_{k+1,n}^{*}}z|}{|e^{-i\theta_k}e^{C_{k+1,n}^{*}}z-1|(|e^{C_{k+1,n}^{*}}z|-1)}.
\intertext{Since $|z|>r$,}
\left|\delta_{c_k^{*}} \left(e^{-i\theta_k}e^{C_{k+1,n}^{*}}z\right)\right|&\leq \frac{\tilde{\lambda }(c_k^*)^{\frac{3}{2}}e^{C_{k+1,n}^{*}}r}{(e^{C_{k+1,n}^{*}}r-1)^2}.
\intertext{Note that $\tilde{\lambda}$ could equal zero here. So using the estimates on $ e^{C_{k+1,n}^{*}}$ and $\epsilon_{k,n}$ from Lemmas \ref{lemma2.2} and Corollary \ref{3.3} respectively we have the following bound,}
\left|\delta_{c_k^{*}} \left(e^{-i\theta_k}e^{C_{k+1,n}^{*}}z\right)\right|&\leq \lambda(\alpha,c,r)k^{\frac{1}{\alpha}-\frac{3}{2}}n^{\frac{-1}{\alpha}}\\
&\leq \lambda(\alpha,c,r)k^{ -\frac{3}{2}}<\lambda(\alpha,c,r)
\end{align*}
where $\lambda(\alpha,c,r)$ is a constant dependent on $\alpha$, $c$ and $r$.
\end{proof}
Note that we will need the intermediate bound in a later proof. Now using Definition \ref{etadef} we see,
$$\;\;\;\eta(0)=e^{C_{k,n}^{*}}z, \;\;\;\;\;\;\;\;\;\;\;\;\;\;\;\eta(1)=e^{i \theta_k} f_{c_k^*}(e^{-i\theta_k}e^{C_{k+1,n}^{*}} z)$$
where $f_{c_k^*}(z)$ is defined as in Section 1. Therefore, 
$$|X_{k,n}(z)|=e^{-C_{1,n}^{*}}|\phi_{k-1}(\eta(1))-\phi_{k-1}(\eta(0))|.$$
Before finding pointwise bounds on $X_{k,n}(z)$ and $T_n(z)$, we first find pointwise bounds on elements of $\eta_{k,n}(s,z)$ and its derivative.
\begin{lemma}\label{3.6}
For $\eta_{k,n}(s,z)$ defined in (\ref{etadef}), for each $z\in \mathbb{C}$ with $|z|>r$ and each $0\leq s\leq 1$, the following pointwise bound holds,
\begin{align*}
\left|\exp\left(s\left(\frac{2c_k^{*}}{e^{-i\theta_k}e^{C_{k+1,n}^{*}}z-1}+\delta_{c_k^{*}} \left(e^{-i\theta_k}e^{C_{k+1,n}^{*}}z\right)\right)\right)\right|\leq \lambda(\alpha,c,r)
\end{align*}
where $\lambda(\alpha,c,r)$ is a constant dependent on $\alpha$, $c$ and $r$. Furthermore, 
\begin{align*}
|\dot{\eta}_{k,n}(s,z)|\leq \lambda(\alpha,c,r)\Big| \frac{c_k^* e^{C_{k,n}^{*}} z}{ e^{-i\theta_k}e^{C_{k+1,n}^{*}}z -1} \Big|\leq  \lambda(\alpha,c,r)\frac{c_k^* e^{C_{k,n}^{*}} }{ e^{C_{k+1,n}^{*}}r -1}.
\end{align*}
\end{lemma}
\begin{proof}
Let $\lambda(\alpha,c,r)$ be some constant that we allow to vary throughout the proof. First notice that since $c_k^{*}<c$ and $e^{C_{k+1,n}^{*}}|z|>r$ it follows that
$$\left|s\left(\frac{2c_k^{*}}{e^{-i\theta_k}e^{C_{k+1,n}^{*}}z-1}\right)\right|\leq \frac{2c}{r-1}.$$
Therefore as, 
\begin{align*}
&\left|\exp\left(s\left(\frac{2c_k^{*}}{e^{-i\theta_k}e^{C_{k+1,n}^{*}}z-1}+\delta_{c_k^{*}} \left(e^{-i\theta_k}e^{C_{k+1,n}^{*}}z\right)\right)\right)\right|\\&\leq \exp\left(\left|\frac{2c_k^{*}}{e^{-i\theta_k}e^{C_{k+1,n}^{*}}z-1}\right|+\left|\delta_{c_{k}^{*}} \left(e^{-i\theta_k}e^{C_{k+1,n}^{*}}z\right)\right|\right)
\end{align*}
we use the bound above along with Lemma \ref{deltabound} to reach the following bound
\begin{align*}
\left|\exp\left(s\left(\frac{2c_k^{*}}{e^{-i\theta_k}e^{C_{k+1,n}^{*}}z-1}+\delta_{c_k^{*}} \left(e^{-i\theta_k}e^{C_{k+1,n}^{*}}z\right)\right)\right)\right|&\leq \exp\left(\frac{2c}{r-1}+\lambda(\alpha,c,r)\right)\\
&=\lambda(\alpha,c,r).
\end{align*}
Now consider $\dot{\eta}_{k,n}(s,z)$. Recalling that
$$\eta_{k,n}(s,z)=e^{C_{k,n}^{*}}z\exp\left(s\left(\frac{2c_k^{*}}{e^{-i\theta_k}e^{C_{k+1,n}^{*}}z-1}+\delta_{c_k^{*}} \left(e^{-i\theta_k}e^{C_{k+1,n}^{*}}z\right)\right)\right)$$
we see that
\begin{align*}
|\dot{\eta}_{k,n}(s,z)|&\leq \left|\left(\frac{2c_k^{*}}{e^{-i\theta_k}e^{C_{k+1,n}^{*}}z-1}+\delta_{c_k^{*}} \left(e^{-i\theta_k}e^{C_{k+1,n}^{*}}z\right)\right)\right||\eta_{k,n}(s,z)|.
\intertext{Then using the bound we found above,}
|\dot{\eta}_{k,n}(s,z)|&\leq \lambda(\alpha,c,r)|e^{C_{k,n}^{*}}z|\left(\left|\frac{2c_k^{*}}{e^{-i\theta_k}e^{C_{k+1,n}^{*}}z-1}\right|+\left|\delta_{c_k^{*}} \left(e^{-i\theta_k}e^{C_{k+1,n}^{*}}z\right)\right|\right)
\intertext{where $\lambda(\alpha,c,r)$ is some constant. Now using the fact that $|z|>r$ and the bound from the proof of Lemma \ref{deltabound} we see that}
|\dot{\eta}_{k,n}(s,z)|&\leq\lambda(\alpha,c,r)\Big| \frac{2c_k^* e^{C_{k,n}^{*}} z}{ e^{-i\theta_k}e^{C_{k+1,n}^{*}}z -1} \Big|\leq  \lambda(\alpha,c,r)\frac{2c_k^* e^{C_{k,n}^{*}} }{ e^{C_{k+1,n}^{*}}r -1}.
\end{align*}
where the second inequality follows by using that $|z|>r$ again.
\end{proof}
Now we can use the bounds above to give us a pointwise bound on $X_{k,n}(z)$. We will use the following distortion theorem in the proof \cite{pommerenke1975univalent}.
\begin{theorem}\label{Koebe}
For a function from the exterior disc into the complex plane $F: \Delta \to \mathbb{C}$ that is univalent except for a simple pole at $\infty$ and Laurent expansion of the form 
$$F(z)=z+a_0+\sum^{\infty}_{n=1}a_n z^{-n}$$
we have the estimate 
$$\frac{|z|^2-1}{|z|^2}\leq |F'(z)|\leq \frac{|z|^2}{|z|^2 -1}\leq \frac{|z|}{|z| -1} \; \; \; \; \; z \in \Delta. $$
\end{theorem}
Our bound on  $X_{k,n}(z)$ is given by the following lemma.
\begin{lemma}\label{4.6}
For the sequence $\{X_{k,n}(z)\}_{k=0}^n$ and corresponding filtration $\mathcal{F}_k$ defined as above, and for a fixed $|z|>r$, the following property is satified for all $0<k \leq n$;
$$|X_{k,n}(z)|<\lambda(\alpha,c,r)\frac{c_k^{*}}{e^{C_{k+1,n}^{*}}r -1}$$where $\lambda(\alpha,c,r)$ is a constant dependent on $\alpha$, $c$ and $r$.
Furthermore, for $0<\alpha\leq1$,
$$\sup_{k\leq n}|X_{k,n}(z)|<\lambda(\alpha,c,r)\frac{1}{n}$$
and for $\alpha>1,$
$$\sup_{k\leq n}|X_{k,n}(z)|<\lambda(\alpha,c,r)\frac{1}{n^{\frac{1}{\alpha}}}.$$

\end{lemma}
\begin{proof}
By definition
\begin{align*}
|X_{k,n}(z)|&=e^{-C_{1,n}^{*}}|\phi_{k-1}(\eta(1))-\phi_{k-1}(\eta(0))|.\intertext{Hence,}
|X_{k,n}(z)|&=e^{-C_{1,n}^{*}}\left|\int_0^1 \phi_{k-1}'(\eta_{k,n}(s,z))\; \dot{\eta}_{k,n}(s) ds\right|\\
&\leq e^{-C_{1,n}^{*}}\int_0^1 \left|\phi_{k-1}'(\eta_{k,n}(s,z))\right|| \dot{\eta}_{k,n}(s) |\;  ds 
\end{align*}
Using Lemma \ref{3.6} we have,
$$|\dot{\eta}_{k,n}(s)|\leq \lambda(\alpha,c,r)\frac{2c_k^* e^{C_{k,n}^{*}} r}{ e^{C_{k+1,n}^{*}}r -1}.$$
where $\lambda(\alpha,c,r)$ is a non-zero constant that will vary throughout this proof. Moreover, we can find a bound on $\int_0^1 \left|\phi_{k-1}'(\eta_{k,n}(s,z))\right|\;  ds $ using Theorem \ref{Koebe},
$$\int_0^1 \left|\phi_{k-1}'(\eta_{k,n}(s,z))\right|\;  ds < e^{C_{1,k-1}^{*}}\sup_{0<s<1} \frac{| \eta_{k,n}(s,z)|}{| \eta_{k,n}(s,z)| -1}.$$
Note that in order to apply the distortion theorem to our function $\phi_{k-1}$ we had to rescale by a factor of $e^{C_{1,k-1}^{*}}$. It is easy to show that $\inf_{0\leq s\leq 1}|\eta_{k,n}(s,z)|\geq|z|$ and therefore for $|z|>r$,
$$\int_0^1 \left|\phi_{k-1}'(\eta_{k,n}(s,z))\right|\;  ds < e^{C_{1,k-1}^{*}} \frac{r}{ r -1}.$$
Thus, by compiling the bounds above,
\begin{align*}
|X_{k,n}(z)|&<\lambda(\alpha,c,r)e^{-C_{1,n}^{*}}\frac{e^{C_{1,k-1}^{*}} r }{ r -1}\frac{2c_k^* e^{C_{k,n}^{*}} r}{ e^{C_{k+1,n}^{*}}r -1}\\
&<\lambda(\alpha,c,r)\frac{c_k^{*}}{e^{C_{k+1,n}^{*}}r -1}.\intertext{Using the estimates in Lemma \ref{lemma2.2} and Corollary \ref{3.3} we have,}
|X_{k,n}(z)|&<\lambda(\alpha,c,r)k^{\frac{1}{\alpha}-1} n^{-\frac{1}{\alpha}}.
\end{align*}
First consider the case where $0<\alpha\leq 1$. Then $\frac{1-\alpha}{\alpha}\geq 0.$ Hence, it is clear that the maximum occurs when $k=n$ and thus
$$\sup_{k\leq n}|X_{k,n}(z)|<\lambda(\alpha,c,r)\frac{1}{n}$$
However, when $\alpha>1$, $k^{\frac{1-\alpha}{\alpha}}<1$ , so
$$\sup_{k\leq n}|X_{k,n}(z)|<\lambda(\alpha,c,r)\frac{1}{n^{\frac{1}{\alpha}}}$$
where $\lambda(\alpha,c,r)$ is a constant dependent on $\alpha$, $c$ and $r$.
\end{proof}
It is now clear to see that as $n$ approaches infinity the bound on $X_{k,n}(z)$ approaches zero pointwise. 

\begin{corollary}
For $X_{k,n}(z)$ defined as above;
$$\lim_{n\to \infty} \sup_{k\leq n} |X_{k,n}(z)|=0$$
\end{corollary}

Now we want to calculate a bound on the variation $T_n(z)= \sum_{k=1}^n\mathbb{E}(|X_{k,n}(z)|^2 | \mathcal{F}_{k-1}).$ This is given by the following lemma. 
\begin{lemma}\label{boundt}
The following inequality holds for sufficiently large $n$. If $0<\alpha<2$,
$$T_n(z)\leq \lambda(\alpha,c,r)\frac{1}{n}$$
where $\lambda(\alpha,c,r)>0$ is some constant.
\end{lemma}
\begin{proof}
First let us look at $|X_{k,n}(z)|^2$. As before we can bound 
$$|X_{k,n}(z)|^2<e^{-2C_{1,n}^{*}}\Big|\int_0^1 \phi_{k-1}'(\eta_{k,n}(s,z))\;  ds\Big|^2 \sup_{0\leq s\leq 1}| \dot{\eta_{k,n}}(s)|^2.$$
Therefore,
$$\mathbb{E}(|X_{k,n}(z)|^2 \; |\; \mathcal{F}_{k-1})\leq e^{-2C_{1,n}^{*}} \mathbb{E}\left( \Big|\int_0^1 \phi_{k-1}'(\eta_{k,n}(s,z))\;  ds\Big|^2 \sup_{0\leq s\leq 1}| \dot{\eta_{k,n}}(s) |^2 \; |\; \mathcal{F}_{k-1} \right).$$
We can find an upper bound on the integral using a distortion theorem again and then remove it from the expectation. By above,
$$\Big|\int_0^1 \phi_{k-1}'(\eta_{k,n}(s,z))\;  ds\Big|^2< e^{2C_{1,k-1}^{*}} \frac{r^2}{ (r -1)^2}.$$
So all that remains to calculate is $\mathbb{E}(\sup_{0\leq s\leq 1}|\dot{\eta_{k,n}(s,z)}|^2 | \mathcal{F}_{k-1})$. Firstly by Lemma \ref{3.6}, for all $0\leq s\leq 1$,
\begin{align*}
|\dot{\eta_{k,n}}(s)|&\leq  \lambda(\alpha,c,r)\frac{c_k^* e^{C_{k,n}^{*}} }{ e^{C_{k+1,n}^{*}}r -1}.\intertext{Then let $w=e^{C_{k+1,n}^{*}}r$ and so}
|\dot{\eta_{k,n}}(s,z)|&\leq \lambda(\alpha,c,r)\frac{c_k^* \; e^{c_k^*} w}{e^{-i\theta_k}w -1}.
\end{align*}
Moreover, since the $c_k^*$ are predetermined, the only randomness here comes from the $\theta_k$ and thus,
$$\mathbb{E}(\sup_{0\leq s\leq 1}|\dot{\eta_{k,n}}(s)|^2 \; |\; \mathcal{F}_{k-1})\leq4(c_k^*)^2e^{2c_k^*}\int_0^{2\pi} \frac{|w|^2}{|e^{-i\theta}w -1|^2} d\theta.$$
It is easily shown that for $w \in \mathbb{C}$,
$$\int_0^{2\pi} \frac{|w|^2}{|e^{-i\theta}w -1|^2} d\theta\leq \frac{6|w|}{ |w|-1}.$$
Therefore,
\begin{align*}
\mathbb{E}(\sup_{0\leq s\leq 1}|\dot{\eta_{k,n}}(s)|^2 | \mathcal{F}_{k-1})&\leq 24(c_k^*)^2e^{2c_k^*}\frac{r e^{C_{k+1,n}^{*}}}{ r e^{C_{k+1,n}^{*}}-1}.
\intertext{It is clear for all $k\leq n$, $c_k^*<c$, therefore,}
\mathbb{E}(\sup_{0\leq s\leq 1}|\dot{\eta_{k,n}}(s)|^2 \; |\; \mathcal{F}_{k-1})&\leq 24e^{2c}(c_k^*)^2\frac{r e^{C_{k,n}^{*}}}{r e^{C_{k+1,n}^{*}}-1}.
\end{align*}
Finally we can use the bound $\frac{1}{{ re^{C_{k+1,n}^{*}}-1}}\leq \frac{1} {re^{C_{k,n}^{*}}-1}\frac{e^cr}{r-1}$ and bring together the previous bounds to reach the following bound on $T_n(z)$. Let $\lambda(\alpha,c,r)>0$ be some constant that will vary throughout. Then,
\begin{align*}
T_n(z) &\leq \lambda(\alpha,c,r) \sum_{k=1}^n \left(e^{-2C_{1,n}^{*}} \; \; e^{2C_{1,k-1}^{*}} \; \; (c_k^*)^2\frac{e^{C_{k,n}^{*}}}{ re^{C_{k,n}^{*}}-1}\right)\\
&\leq \lambda(\alpha,c,r)\sum_{k=1}^n(c_k^*)^2 \frac{e^{-C_{k,n}^{*}}}{\left(e^{C_{k,n}^{*}}r-1\right)}\\
&\leq \lambda(\alpha,c,r)\sum_{k=1}^n(c_k^*)^2 e^{-2C_{k,n}^{*}}.
\intertext{We can substitute in the known values for $c_k^*$ and $C_{k,n}^{*}$to reach the following bound on $T_n(z)$, 
}
T_n(z)&\leq \lambda(\alpha,c,r)\sum_{k=1}^n\left(\frac{c}{1+\alpha c(k-1)}\right)^2\left(\frac{1+\alpha c(k-1)}{1+\alpha c n}\right)^{\frac{2}{\alpha}(1+\epsilon_{k,n})}.
\intertext{Let $x=\frac{1+\alpha c(k-1)}{1+\alpha c n}$. Then,}
T_n(z)&\leq \lambda(\alpha,c,r)\frac{1}{1+\alpha c n}\int_{\frac{1}{1+\alpha c n}}^1 x^{\frac{2}{\alpha}-2} dx.
\end{align*} 
This integral is bounded above by a constant if $0<\alpha<2$. Therefore we bound above by
$$T_n(z)\leq \lambda(\alpha,c,r) \frac{1}{n}.$$
\end{proof}
Moreover since $T_n(z)\geq 0$, we have the following corollary.
\begin{corollary}
For $0<\alpha<2$,
$$\lim_{n \to \infty}T_n(z)=0$$
\end{corollary}

\subsection{Results}
We are now in a position to analyse the limiting structure of the map $\phi_n$ as $n\to \infty$ for $0<\alpha<2$. Our aim is to use the bounds on the increments $X_{k,n}(z)$ and $T_n(z)$ found in the previous section to produce a pointwise estimate on the difference  between the cluster map and the disk of capacity $e^{C_{k,n}^{*}}$. In order to do so we will apply the following theorem of Freedman \cite{freedman1975tail}.
\begin{theorem}[Freedman]\label{freedman}
Suppose $X_{k,n}$ is $\mathcal{F}$-measurable and $\mathbb{E} \{X_{k,n} \; | \; \mathcal{F}_{k-1} \}=0$ and define $M_n=\sum_{k=1}^n X_{k,n}$ and $T_n=\sum_{k=1}^n \mathrm{Var}\{X_{k,n} \; | \; \mathcal{F}_{k-1} \}$. Let $M$ be a positive real number and suppose $\mathbb{P} \{ |X_{k,n}| \leq M \; | \; k \leq n \}=1$. Then for all positive numbers $a$ and $b$,
$$\mathbb{P} \{ M_n \geq a \; \; \text{and} \; \; T_n\leq b \; \text{for some } n>0 \} \leq \exp\left[\frac{-a^2}{2(Ma+b)}\right].$$
\end{theorem}
Recall, $e^{-C_{1,n}^{*}}\phi_n(z)-z=\sum_{k=1}^{n} X_{k,n}(z)$. Hence, we can now apply Theorem \ref{freedman} to our cluster to obtain pointwise results for $0<\alpha<2$.
\newpage
\begin{theorem}\label{4.8}
Let $c_i^*$ and $\phi_k$ be defined as above. Then for $0<\alpha<2$, and any positive real number $a\leq \frac{\log(n)}{\sqrt{n}}$ and $n$ sufficiently large,
$$ \mathbb{P}\left(|e^{-C_{1,n}^{*}}\phi_n(z)-z|>a\right)\leq 2e^{\frac{-a^2n}{\lambda(\alpha,c,r)}}$$
for some strictly positive constant $\lambda(\alpha,c,r)$. Therefore, for all $0<\alpha<2$ and \newline $\frac{1}{\sqrt{ n}}\ll a\leq \frac{\log(n)}{\sqrt{n}}$ and for all $z\in \mathbb{C	}$ with $|z|>1$,
 
$$\lim_{n \to \infty} \mathbb{P}\left(|e^{-C_{1,n}^*}\phi_n(z)-z|>a\right)=0.$$
\end{theorem}
\begin{proof}
First note, we have shown in Lemma \ref{lemma4.2}, $\mathbb{E}(X_{k,n}(z)|\mathcal{F}_{k-1})=0$ where$$X_{k,n}(z)=e^{-C_{1,n}^{*}}\left(\phi_k\left(e^{C_{k+1,n}^*}z\right)-\phi_{k-1}\left(e^{C_{k,n}^{*}}z\right)\right).$$
Recall, $M_n(z)= \sum_{k=1}^n X_{k,n}(z)$, and note that we can split $M_n$ into real and imaginary parts, thus, 
$$\mathbb{P}\left(\big|M_n \big| >a\right)\leq \mathbb{P}\left(\Re(M_n) >\frac{a}{\sqrt{2}}\right)+\mathbb{P}\left(\Im(M_n) >\frac{a}{\sqrt{2}}\right).$$
Moreover, 
$$\sup_{k\leq n} \Re(X_{k,n}(z))<\sup_{k\leq n} |X_{k,n}(z)|$$
$$\sup_{k\leq n} \Im(X_{k,n}(z))<\sup_{k\leq n} |X_{k,n}(z)|.$$
It is easy to see that both $\Re(X_{k,n}(z))$ and $\Im(X_{k,n}(z))$ satisfy the property that the expectation with respect to the filtration is zero and so by Theorem \ref{freedman}, for any positive real number $a$,
\begin{align*}
\mathbb{P}\left(\big|\sum_{k=1}^n X_{k,n}(z) \big|\geq a\right)&\leq \mathbb{P}\left(\Re\left(\sum_{k=1}^n X_{k,n}(z) \right)\geq \frac{a}{\sqrt{2}}\right)+\mathbb{P}\left(\Im\left(\sum_{k=1}^n X_{k,n}(z)\right) \geq \frac{a}{\sqrt{2}}\right)\\
& \leq 2\exp\left[ \frac{-a^2}{4(b_X(k,n)\frac{a}{\sqrt{2}}+b_T(n))} \right]
\end{align*}
where $b_X(k,n)$, $b_T(n)$ are the bounds on $|X_{k,n}(z)|$ and $T_n(z)$ respectively. We first deal with the case that $0<\alpha\leq 1$. In Lemma \ref{4.6} we have seen
$$\sup_{k\leq n}|X_{k,n}(z)|< \lambda^1(\alpha,c,r)\frac{1}{n}$$
for some positive constant $\lambda^1(\alpha,c,r)$ and by Lemma \ref{boundt},
$$T_n(z)\leq \lambda^2(\alpha,c,r)\frac{1}{n}$$
for some positive constant $\lambda^2(\alpha,c,r)$. Therefore, 
$$\mathbb{P}\left(|e^{-\sum_{i=1}^nc_i^*}\phi_n(z)-z|>a\right)\leq 2e^{\frac{-a^2n}{4\left(\lambda^1(\alpha,c,r) \frac{a}{\sqrt{2}}+\lambda^2(\alpha,c,r)\right)}}.$$
But for $n$ sufficiently large, $\lambda^1(\alpha,c,r)\frac{a}{\sqrt{2}} \leq \lambda^2(\alpha,c,r)$  so let $\lambda(\alpha,c,r)=8\lambda^2(\alpha,c,r)$ then
$$ \mathbb{P}\left(|e^{-\sum_{i=1}^nc_i^*}\phi_n(z)-z|>a\right)\leq 2e^{\frac{-a^2 n}{\lambda(\alpha,c,r)}}.$$
Now for $1<\alpha<2,$
$$\sup_{k\leq n}|X_{k,n}(z)|<\lambda^1(\alpha,c,r)\frac{1}{n^{\frac{1}{\alpha}}}$$ for some positive constant $\lambda^1(\alpha,c,r)$ and  
$$T_n(z)\leq \lambda^2(\alpha,c,r)\frac{1}{n}$$
for some positive constant $\lambda^2(\alpha,c,r)$.
Therefore, 
$$\mathbb{P}\left(|e^{-\sum_{i=1}^nc_i^*}\phi_n(z)-z|>a\right)\leq 2e^{\frac{-a^2 n^{\frac{1}{\alpha}}}{4\left(\lambda^1(\alpha,c,r)\frac{a}{\sqrt{2}}+\lambda^2(\alpha,c,r)n^{\frac{1-\alpha}{\alpha}}\right)}}.$$
But for $a\leq \frac{\log(n)}{\sqrt{n}}$, and $n$ sufficiently large, $\lambda^1(\alpha,c,r)\frac{a}{\sqrt{2}}\leq \lambda^2(\alpha,c,r) n^{\frac{1-\alpha}{\alpha}}$. Therefore, using the same $\lambda(\alpha,c,r)$ as above,
$$ \mathbb{P}\left(|e^{-\sum_{i=1}^nc_i^*}\phi_n(z)-z|>a\right)\leq 2e^{\frac{-a^2 n}{\lambda(\alpha, r,c)}}.$$
So for all $0<\alpha<2$,
$$ \mathbb{P}\left(|e^{-\sum_{i=1}^nc_i^*}\phi_n(z)-z|>a\right)\leq 2e^{\frac{-a^2 n}{\lambda(\alpha, r,c)}}.$$
Therefore for $\frac{1}{\sqrt{ n}}\ll a\leq \frac{\log(n)}{\sqrt{n}}$,
 
$$\lim_{n \to \infty} \mathbb{P}\left(|e^{-C_{1,n}^*}\phi_n(z)-z|>a\right)=0.$$
\end{proof}

\section{Uniform convergence in the exterior disk for $0<\alpha<2$}
So far we have seen that when evaluated at a fixed point our map looks like a disk. Our aim now is to show that if we map from a disc of fixed radius then all points on the exterior disk will satisfy the same property. Our aim of this section will be to prove the following theorem. 
\begin{thm2}
For $0<\alpha<2$, let the map $\phi_n$ be defined as above with $c_n^{*}$ as defined in (\ref{eqn3}) and $\theta_n$ i.i.d, uniform on $[0, 2\pi]$. Then for any $r>1$ we have the following inequality $$\mathbb{P}\left(\sup_{|z|\geq r}|e^{-\sum_{i=1}^nc_i^*}\phi_n(z)-z|>\frac{\log(n)}{\sqrt{n}} \right)
<\lambda^1(\alpha,c,r)e^{-\frac{\log(n)^2}{\lambda^2(\alpha,c,r)}}$$
where $\lambda^1(\alpha,c,r), \; \lambda^2(\alpha,c,r)>0$ are constants. Hence, by Borel Cantelli,
$$\mathbb{P}\left(\limsup_{n\to \infty}\left\lbrace\sup_{|z|\geq r}|e^{-\sum_{i=1}^nc_i^*}\phi_n(z)-z|>\frac{\log n}{\sqrt{n}}\right\rbrace\right)=0.$$
\end{thm2}
The proof of the theorem will be constructed as follows. We will show that for a finite number of equally spaced points along the circle $|z|=r$ the inequality holds. Then we will show that between these points the probability that the difference between the maps when evaluated at these points is sufficiently small. First define
$$M_n(z,w):=M_n(z)-M_n(w)$$
with $M_n(z)$ defined in equation (\ref{mart}). Then we must choose the spacing between the finite set of points. With the choice of $\alpha$ and $c$ fixed we choose points, on a radius $|z|=r$, to be equally spaced at angles $\frac{2 \pi}{L_{r,n}}$ where 
$$L_{r,n}= \gamma(\alpha,c,r) n^{\frac{3}{2}}$$
and $\gamma(\alpha,c,r)$ is a constant,
\begin{equation}\label{eqn2}
\gamma(\alpha,c,r)=4\pi r \frac{1}{c}(e^c+1)(1+\alpha c)(1+\alpha e^{\alpha c})\left(\log \left(\frac{r}{r-1}\right)+1\right)\left(\log(1+\alpha c)+1\right).
\end{equation} The reason for this choice of spacing will become clear in the proof of the lemmas that follow. We start by proving that we can find a finite number of equally spaced points, with the above spacing along the circle $|z|=r$, such that the inequality in Theorem \ref{a} holds.
\begin{lemma}\label{5.2}
Let $\{z_i\}_{i=1}^{L_{r,n}}$ be defined as finite set of points on the boundary of the unit circle of radius $|z|=r$ with equally spaced at angles $\frac{2 \pi}{L_{r,n}}$ and $L_{r,n}$ defined as above. Then, for sufficiently large $n$, we have the following inequality $$\mathbb{P}\left(\exists i: |M_n(z_i)|>\frac{1}{2}\sqrt{\frac{(\log(1+\alpha c n))^2}{(1+\alpha c n)}}\right)<\lambda^1(\alpha, c,r)e^{\frac{-(\log(1+\alpha c n))^2}{\lambda^2(\alpha, c,r)}}$$
where $\lambda^1(\alpha, c, r),\lambda^2(\alpha, c,r)>0$ are constants.
\end{lemma}
\begin{proof}
We have shown using Theorem \ref{4.8} that for $0<\alpha<2$ and $a\leq\frac{\log n}{\sqrt{n}}$ and for any $1\leq i\leq L_{r,n}$
\begin{align*}
\mathbb{P}\left( |M_n(z_i)|>\frac{a}{2}\right)\leq & 2e^{\frac{-a^2n}{\lambda(\alpha,c,r)}}
\intertext{for some constant $\lambda(\alpha,c,r)>0$. Therefore,}
\mathbb{P}\left(\exists i: |M_n(z_i)|>\frac{a}{2}\right)<&2\sum_{k=1}^{L_{r,n}} e^{\frac{-a^2n}{\lambda(\alpha,c,r)}}.
\end{align*}
So let $a^2=\frac{\log(n)^2 }{n}$. Then,
$$\mathbb{P}\left(\exists i: |M_n(z_i)|>\frac{\log n}{2\sqrt{n}}\right)\leq 2\sum_{k=1}^{L_{r,n}} e^{\frac{-\log(1+\alpha c n)^2}{\lambda(\alpha,c,r)}}.$$
The terms in the sum have no dependence on $k$ and as such we can find an upper bound,
\begin{align*}
\mathbb{P}\left(\exists i: |M_n(z_i)|>\frac{\log n}{2\sqrt{n}}\right)&\leq 2L_{r,n}e^{\frac{-\log( n)^2}{\lambda(\alpha,c,r)}}\\
&=\gamma(\alpha,c,r) n^{\frac{3}{2}}e^{\frac{-\log( n)^2}{\lambda(\alpha,c,r)}}
\end{align*}
where $\gamma(\alpha,c,r)>0$ is the constant defined in equation (\ref{eqn2}).
Let $\lambda^1(\alpha, c, r)=\gamma(\alpha,c,r)$, then 
$$\mathbb{P}\left(\exists i: |M_n(z_i)|>\frac{\log n}{2\sqrt{n}}\right)\leq \lambda^1(\alpha, c, r)e^{\frac{3}{2}\log n-\frac{\log( n)^2}{\lambda(\alpha,c,r)}}.$$
For sufficiently large $n>e^{3\lambda(\alpha,c,r)}$,
$$\frac{\frac{3}{2}\log n}{\frac{\log( n)^2}{\lambda(\alpha,c,r)}}\leq \frac{1}{2}.$$
Therefore, let $\lambda^2(\alpha, c,r)=2\lambda(\alpha,c,r)$ and then for $n$ sufficiently large,
$$\mathbb{P}\left(\exists i: |M_n(z_i)|>\frac{\log n}{2\sqrt{n}}\right)\leq \lambda^1(\alpha, c, r)e^{\frac{-(\log( n))^2}{\lambda^2(\alpha, c,r)}}$$
with $\lambda^1(\alpha,c,r), \lambda^2(\alpha,c,r)>0$.
\end{proof}
We now prove that for points $w\in \mathbb{C}$ in between the points in the set $\{z_i\}_{i=1}^{L_{r,n}}$ the difference $M_n(z_i,w)$ is negligible.
\begin{lemma}\label{5.3}
For $|z|=|w|=r$ with $\arg(z)=\theta_z$, $\arg(w)=\theta_w$  and $|\theta_z-\theta_w|<\frac{2\pi}{L_{r,n}}$ and $L_{r,n}$ defined as above we have the following bound;
$$|M_n(z,w)|\leq \frac{\log( n)}{2\sqrt{n}}$$ and hence,
$$\mathbb{P}\left(\exists w, z\in \mathbb{C} : |\theta_z-\theta_w|<\frac{2\pi}{L_{r,n}}, \; |M_n(z, w)|>\frac{\log( n)}{2\sqrt{n}}\right)=0.$$
\end{lemma}
\begin{proof}
We want to find a bound on $|M_n(z,w)|$ so we first find a bound on $|X_{k,n}(z,w)|=|X_{k,n}(z)-X_{k,n}(w)|$. 
\begin{align*}
&|X_{k,n}(z,w)|\\
&=e^{-\sum_{k=1}^n c_{i}^{*}}\left|\left(\phi_k\left(e^{C_{k+1,n}^{*}}z\right)-\phi_{k-1}\left(e^{C_{k,n}^{*}}z\right)\right)-\left(\phi_k\left(e^{C_{k+1,n}^{*}}w\right)-\phi_{k-1}\left(e^{C_{k,n}^{*}}w\right)\right)\right|.
\end{align*}
Let $0\leq s,t\leq 1$ and then \begin{align*}
\tau_{k,n}(s)=&e^{C_{k+1,n}^{*}}|z| e^{i(\theta_z s +\theta_w (1-s))}\\
\rho_{k,n}(t)=&e^{C_{k,n}^{*}}|z| e^{i(\theta_z t +\theta_w (1-t))}.
\end{align*}
Thus,
$$|X_{k,n}(z,w)\leq |\phi_k(\tau_{k,n}(1))-\phi_k(\tau_{k,n}(0))|+|\phi_{k-1}(\rho_{k,n}(1))-\phi_{k-1}(\rho_{k,n}(0))|.$$
If we consider the $\tau$ terms in the upper bound, we have
\begin{align*}
|\phi_{k}(\tau_{k,n}(1))-\phi_{k}(\tau_{k,n}(0))|&\leq \int_0^1 \left|\phi_k^{'}(\tau_{k,n} (s)) \right| | \dot{\tau_{k,n}}(s)|ds.
\intertext{Using the distortion theorem \cite{pommerenke1975univalent},}
|\phi_{k}(\tau_{k,n}(1))-\phi_{k}(\tau_{k,n}(0))|&\leq e^{C_{1,k}^{*}} \sup_{0\leq s\leq 1}\frac{|\tau_{k,n}(s)|}{|\tau_{k,n}(s)|-1}e^{C_{k+1,n}^{*}}|\theta_z-\theta_w||z|.
\intertext{Therefore,}
|\phi_{k}(\tau_{k,n}(1))-\phi_{k}(\tau_{k,n}(0))|&\leq e^{C_{1,n}^{*}}|z|^2|\theta_z-\theta_w|\frac{e^{C_{k+1,n}^{*}}}{e^{C_{k+1,n}^{*}}|z|-1}.
\intertext{By a similar argument }
|\phi_{k-1}(\rho_{k,n}(1))-\phi_{k-1}(\rho_{k,n}(0))|& \leq e^{C_{1,n}^{*}}|z|^2|\theta_z-\theta_w|\frac{e^c e^{C_{k+1,n}^{*}}}{e^{C_{k+1,n}^{*}}|z|-1}.
\end{align*}
Therefore using the fact $|z|= r$,
$$|X_{k,n}(z,w)|\leq r^2(e^c+1)|\theta_z-\theta_w|\frac{ e^{C_{k+1,n}^{*}}}{e^{C_{k+1,n}^{*}}r-1}.$$
We can therefore use the approximation $e^{C_{k,n}^{*}}\approx\left(\frac{1+\alpha c n}{1+\alpha c (k-1)}\right)^{\frac{1}{\alpha}}$ and take the sum to write
$$|M_n(z,w)|\leq r^2(e^c+1)|\theta_z-\theta_w| \left|\Large \sum _{k=1}^n\left( \frac{\left(\frac{1+\alpha c n}{1+\alpha c k}\right)^{\frac{1+\epsilon_{k,n}}{\alpha}}}{r\left(\frac{1+\alpha c n}{1+\alpha c k}\right)^{\frac{1+\epsilon_{k,n}}{\alpha}}-1}\right)\right|$$
where $\epsilon_{k,n}$ is the same error term from Section 2. We can use the bound from Corollary \ref{3.3} to remove the $\epsilon_{k,n}$ term, $$\left(\frac{1+\alpha c n}{1+\alpha c k}\right)^{\frac{1}{\alpha}}<\left(\frac{1+\alpha c n}{1+\alpha c k}\right)^{\frac{1+\epsilon_{k,n}}{\alpha}}\leq (1+\alpha c e^{\alpha c})\left(\frac{1+\alpha c n}{1+\alpha c k}\right)^{\frac{1}{\alpha}}$$
Then $x=\left(\frac{1+\alpha c n}{1+\alpha c k}\right)^{\frac{1}{\alpha}}$ and integrating between $x=\left(\frac{1+\alpha c n}{1+\alpha c }\right)^{\frac{1}{\alpha}}$ and $x=1$ gives
\begin{align*}
\left|\Large \sum _{k=1}^n\left( \frac{\left(\frac{1+\alpha c n}{1+\alpha c k}\right)^{\frac{1+\epsilon_{k,n}}{\alpha}}}{r\left(\frac{1+\alpha c n}{1+\alpha c k}\right)^{\frac{1+\epsilon_{k,n}}{\alpha}}-1}\right)\right|&\leq \frac{1}{c}\left|\int_1^{\left(\frac{1+\alpha c n}{1+\alpha c }\right)^{\frac{1}{\alpha}}} \frac{1+\alpha c k}{rx-1}dx\right|\\&\leq \frac{1}{c}(1+\alpha c n)\left|\int_1^{\left(\frac{1+\alpha c n}{1+\alpha c }\right)^{\frac{1}{\alpha}}} \frac{1}{rx-1}dx\right|.
\end{align*}
Thus,
\begin{align*}
\left|\Large \sum _{k=1}^n\left( \frac{\left(\frac{1+\alpha c n}{1+\alpha c k}\right)^{\frac{1+\epsilon_{k,n}}{\alpha}}}{r\left(\frac{1+\alpha c n}{1+\alpha c k}\right)^{\frac{1+\epsilon_{k,n}}{\alpha}}-1}\right)\right|&\leq \frac{1}{cr}(1+\alpha cn)\left|\log\left(\frac{r-1}{r\left(\frac{1+\alpha c n}{1+\alpha c }\right)^{\frac{1}{\alpha}}-1}\right)\right|\\
&\leq\frac{1}{cr}(1+\alpha cn)\log\left(\frac{r\left(1+\alpha c n\right)^{\frac{1}{\alpha}}}{r-1}\right).
\end{align*}
Therefore,
\begin{align*}
|M_n(z,w)|&\leq \frac{\gamma(\alpha,c,r)}{4\pi}|\theta_z-\theta_w| n\log n
\intertext{where $\gamma(\alpha,c,r)$ is the constant defined in equation (\ref{eqn2}). Then we use the fact that $|\theta_z-\theta_w|=\frac{2 \pi}{L_{r,n}}$ and write}
|M_n(z,w)|&\leq \frac{\log n}{2\sqrt{n}}.
\end{align*}
So, $$\mathbb{P}\left(\exists w, z\in \mathbb{C} : |\theta_z-\theta_w|<\frac{2\pi}{L_{r,n}}, \; |M_n(z, w)|>\frac{\log n}{2\sqrt{n}}\right)=0.$$

\end{proof}

So we can combine these two lemmas to give our proof of Theorem \ref{a}.
\begin{proof}[Proof of Theorem \ref{a}]
As in the previous two lemmas we separate the circle into points $\frac{2 \pi}{L_{r,n}}$ apart. We can then form the following bound;
\begin{align*}
&\mathbb{P}\left(\sup_{|z|=r}|e^{-C_{1,n}^*}\phi_n(z)-z|>\frac{\log n}{\sqrt{n}}\right)\\&\leq \mathbb{P}\left(\exists i: |M_n(z_i)|>\frac{1}{2}\frac{\log n}{\sqrt{n}}\right) \\&+ \mathbb{P}\left(\exists w, z\in \mathbb{C} : |\theta_z-\theta_w|<\frac{2\pi}{L_{r,n}}, \; M_n(z, w)>\frac{1}{2}\frac{\log n}{\sqrt{n}}\right).
\end{align*}
Using Lemmas \ref{5.2} and \ref{5.3} we see,
$$\mathbb{P}\left(\sup_{|z|=r}|e^{-C_{1,n}^*}\phi_n(z)-z|>\frac{\log n}{\sqrt{n}}\right)\leq \lambda^1(\alpha, c, r)e^{\frac{-(\log(n))^2}{\lambda^2(\alpha, c,r)}}$$
where $\lambda^1(\alpha, c, r),\lambda^2(\alpha, c,r)>0$ are constants. Then using the maximum modulus principle we see that that the maxiumum occurs on the boundary and so,
$$\mathbb{P}\left(\sup_{|z|\geq r}|e^{-C_{1,n}^*}\phi_n(z)-z|>\frac{\log n}{\sqrt{n}}\right)\leq \lambda^1(\alpha, c, r)e^{\frac{-(\log(n))^2}{\lambda^2(\alpha, c,r)}}.$$
It is clear to see the upper bound is summable and hence by a Borel Cantelli argument,
$$\mathbb{P}\left(\limsup_{n\to \infty}\sup_{|z|\geq r}|e^{-C_{1,n}^*}\phi_n(z)-z|>\frac{\log n}{\sqrt{n}}\right)=0.$$
\end{proof}

\section{Fluctuations for $0<\alpha<2$}
\subsection{Discarding the lower order terms}
In the previous sections we have seen that, by using the result of Freedman \cite{freedman1975tail}, we have convergence to a disk in the exterior disk. Now we would like to see how much we fluctuate from this disk. To do so we aim to produce a central limit theorem that will tell us what the distribution of the fluctuations is. Up until this point we have used
$$X_{k,n}(z)=e^{-C_{1,n}^{*}}\left(\phi_k\left(e^{-C_{k+1,n}^{*}}z\right)-\phi_{k-1}\left(e^{-C_{k,n}^{*}}z\right)\right).$$ 
We aim to prove that the fluctuations are of order $\sqrt{n}$. First, we want to show we can discard the lower order terms of the increments $X_{k,n}(z)$ in order to simplify the calculation of the fluctuations. Therefore, we introduce the rescaled increment,
$$\mathcal{X}_{k,n}(z)=\frac{2 c_k^* \sqrt{n}z}{ e^{-i \theta_k}e^{C_{k+1,n}^{*}}z-1}. $$
The following lemma shows that we can discard the lower order terms.  
\begin{lemma}
Let $Y_{k,n}(z)=\sqrt{n}X_{k,n}(z)-\mathcal{X}_{k,n}(z)$. Then if $0<\alpha<2$, for any $\epsilon>0$ and $r>1$,
$$\mathbb{P}\left(\lim_{n\to\infty}\sup_{|z|>r}\left|\sum_{k=1}^n Y_{k,n}(z)\right| >\epsilon\right)=0$$

\end{lemma}
\begin{proof}
Fix some $r>1$. Then in Theorem \ref{a} we showed that for $|z|>r$, 
$$\mathbb{P}\left(\limsup_{n\to \infty}\left\lbrace\sup_{|z|\geq r}|e^{-C_{1,n}^*}\phi_n(z)-z|>\frac{\log n}{\sqrt{n}}\right\rbrace\right)=0.$$
Denote the event,  $$\omega(r)=\left\lbrace \limsup_{n\to \infty}\left\lbrace\sup_{|z|\geq r}|e^{-C_{1,n}^*}\phi_n(z)-z|\leq\frac{\log n}{\sqrt{n}}\right\rbrace\right\rbrace.$$ Now  choose $r'=\frac{r+1}{2}$, then,
\begin{align*}
\mathbb{P}\left(\lim_{n\to\infty}\sup_{|z|>r}\left|\sum_{k=1}^n Y_{k,n}(z)\right| <\epsilon\right)&=\mathbb{P}\left(\lim_{n\to\infty}\sup_{|z|>r}\left|\sum_{k=1}^n Y_{k,n}(z)\right| <\epsilon \; \middle| \;\omega(r')\right)\mathbb{P}\left(\omega(r')\right)\\&+\mathbb{P}\left(\lim_{n\to\infty}\sup_{|z|>r}\left|\sum_{k=1}^n Y_{k,n}(z)\right| <\epsilon \; \middle| \; \omega(r')^{c}\right)\mathbb{P}\left(\omega(r')^c\right).
\end{align*}
We have shown that $\mathbb{P}\left(\omega(r')\right)=1$. Therefore,
\begin{align*}
\mathbb{P}\left(\lim_{n\to\infty}\sup_{|z|>r}\left|\sum_{k=1}^n Y_{k,n}(z)\right| <\epsilon\right)=\mathbb{P}\left(\lim_{n\to\infty}\sup_{|z|>r}\left|\sum_{k=1}^n Y_{k,n}(z)\right| <\epsilon \;\middle| \; \omega(r')\right).
\end{align*}
We first calculate a bound on $|Y_{k,n}(z)|$. Let \begin{align*}
\widetilde{X}_{k,n}(z)=\sqrt{n}e^{-C_{k,n}^{*}}\int_0^1\dot{\eta}_{k,n}(s,z)ds
\end{align*} 
where $\eta_{k,n}(s,z)$ is defined as in Section 3. Then, \begin{align*}
\sqrt{n} X_{k,n}(z)-\widetilde{X}_{k,n}(z)=\sqrt{n}e^{-C_{1,n}^{*}}\left(\int_0^1 \dot{\eta}_{k,n}(s,z)\left(\phi_{k-1}(\eta_{k,n}(s,z))-e^{C_{1,k-1}^{*}}\right)ds\right) .
\end{align*}
But for $|z|>r'$ on the event $\omega(r')$,
$$|e^{-C_{1,k-1}^*} \phi_n(z)-z|<\frac{\log(k-1)}{\sqrt{k-1}}.$$
Then let $g(z)=e^{-C_{1,k-1}^*} \phi_n(z)-z$. The  map $g$ is holomorphic on the closed disc $|\zeta -z|<R:=|z|-r'.$
So by Cauchy's theorem, for $0<\alpha<2$, 
$$g'(z)=\frac{1}{2\pi i} \int_{C_R} \frac{g(\zeta)}{(\zeta-z)^2} d\zeta$$
where $C_R$ is the circle of radius $R$ centred at $z$. Therefore,
$$|g'(z)|\leq \frac{1}{(|z|-r')}\frac{\log(k-1)}{\sqrt{k-1}}.$$So on $\omega(r')$, 
\begin{align*}
|X_{k,n}(z)-\widetilde{X}_{k,n}(z)|\leq \sqrt{n}e^{-C_{1,n}^{*}}\left(\int_0^1 \dot{\eta}_{k,n}(s,z)\left(\frac{1}{(|\eta_{k,n}(s,z)|-r')}\frac{e^{C_{1,k-1}^{*}}\log(k-1)}{\sqrt{k-1}} \right)ds\right) .
\end{align*}
Then since, $\inf_{0\leq k\leq n}|\eta_{k,n}(s,z)|\geq | z|,$
\begin{align*}
|X_{k,n}(z)-\widetilde{X}_{k,n}(z)|&\leq \sqrt{n}\frac{1}{ r -r'}e^{-C_{k,n}^{*}} \frac{\log(k-1)}{\sqrt{k-1}}\int_0^1 |\dot{\eta}_{k,n}(s,z)|ds\\
&\leq  \lambda(\alpha,c,r)\sqrt{n}e^{-C_{k,n}^{*}} \frac{\log(k-1)}{\sqrt{k-1}}\frac{c_k^* e^{C_{k,n}^{*}} }{ e^{C_{k+1,n}^{*}}r -1}\\
&\leq  \lambda(\alpha,c,r) \frac{\sqrt{n}}{n^{\frac{1}{\alpha}}} \frac{\log(k)k^{\frac{1}{\alpha}}}{k^{\frac{3}{2}}}.
\end{align*}
Where the second inequality follows from Lemma \ref{3.6}.
Now consider,
\begin{align*}
&|\widetilde{X}_{k,n}(z)-\mathcal{X}_{k,n}(z)|\\
&\leq \sqrt{n}\left|\left(\frac{2c_k^{*}}{e^{-i\theta_k}e^{C_{k+1,n}^{*}}z-1}\right)\left(e^{-C_{k,n}^{*}}\int_{0}^1\eta_{k,n}(s,z)ds-z\right)\right|\\
&+\sqrt{n}\left|\left(e^{-C_{k,n}^{*}}\int_{0}^1\eta_{k,n}(s,z)ds\right)\delta_{c_k^{*}} \left(e^{-i\theta_k}e^{C_{k+1,n}^{*}}z\right)\right|\\
&\leq \sqrt{n}\left(\left(\frac{2c_k^{*}}{e^{C_{k+1,n}^{*}}r-1}\right)\left(r\int_{0}^1\left|e^{x_{k,n}(s)}-1\right|ds\right)+\lambda(\alpha,c,r)
\left|\delta_{c_k^{*}} \left(e^{-i\theta_k}e^{C_{k+1,n}^{*}}z\right)\right|\right)
\end{align*}
where $\lambda(\alpha,c,r)$ is some positive constant that we will vary and $$x_{k,n}(s)=s\left(\frac{2c_k^{*}}{e^{-i\theta_k}e^{C_{k+1,n}^{*}}z-1}+\delta_{c_k^{*}} \left(e^{-i\theta_k}e^{C_{k+1,n}^{*}}z\right)\right).$$ Furthermore,
$$|e^{x_{k,n}(s)}-1|\leq \lambda(\alpha,c,r)|x_{k,n}(s)|\leq \lambda(\alpha,c,r)k^{\frac{1}{\alpha}-1}n^{\frac{-1}{\alpha}}$$
where the second inequality follows from Lemmas \ref{lemma2.2} and \ref{deltabound} and Corollary \ref{3.3}. Hence by using the bound on $\delta_c$ from Lemma \ref{deltabound} we see that,
\begin{align*}
|\widetilde{X}_{k,n}(z)-\mathcal{X}_{k,n}(z)|&\leq \lambda(\alpha,c,r')\sqrt{n}\left(\left(k^{\frac{1}{\alpha}-1}n^{\frac{-1}{\alpha}}\right)^2+k^{\frac{1}{\alpha}-\frac{3}{2}}n^{\frac{-1}{\alpha}}\right).
\intertext{Since $k^{\frac{1}{\alpha}}\leq n^{\frac{1}{\alpha}}$ we have}
|\widetilde{X}_{k,n}(z)-\mathcal{X}_{k,n}(z)|&\leq \lambda(\alpha,c,r)k^{\frac{1}{\alpha}-\frac{3}{2}}n^{\frac{1}{2}-\frac{1}{\alpha}}.
\end{align*}
Therefore, $$|Y_{k,n}(z)|\leq \lambda(\alpha,c,r) \log(n) n^{\frac{1}{2}-\frac{1}{\alpha}}k^{\frac{1}{\alpha}-\frac{3}{2}}.$$
Then we split into cases, if $0<\alpha<\frac{2}{3}$,
\begin{align*}
\sup_{k\leq n}|Y_{k,n}(z)|&\leq \lambda(\alpha,c,r) \frac{\log(n)}{n} \to 0
\intertext{as $n\to \infty$. However, if $\frac{2}{3}<\alpha <2$ then}
\sup_{k\leq n}|Y_{k,n}(z)|&\leq \lambda(\alpha,c,r) \log(n)n^{\frac{1}{2}-\frac{1}{\alpha}} \to 0
\intertext{as $n\to \infty$. Moreover,}
\mathbb{E}(Y_{k,n}(z)|^2 \;| \mathcal{F}_{k-1})&\leq \lambda(\alpha,c,r)\frac{n}{n^{\frac{2}{\alpha}}} \frac{\log(n)^2k^{\frac{2}{\alpha}}}{k^{3}}.
\intertext{thus if $0<\alpha<1$,}
\sum_{k=1}^n \mathbb{E}(|Y_{k,n}(z)|^2 \;| \mathcal{F}_{k-1})&\leq \lambda(\alpha,c,r)\frac{\log(n)^3}{n}\to 0
\intertext{as $n\to \infty$. If $1<\alpha<2$,}
\sum_{k=1}^n \mathbb{E}(|Y_{k,n}(z)|^2 \;| \mathcal{F}_{k-1})&\leq \lambda(\alpha,c,r)\frac{\log(n)^2 n}{n^{\frac{2}{\alpha}}} \to 0
\end{align*}
as $n\to \infty$. Therefore, since $Y_{k,n}(z)$ is also a martingale difference array we can use these bounds to apply the same methods to the difference $Y_{k,n}(z)$ as we did to $X_{k,n}(z)$ in Sections 4 and 5 along with a Borel Cantelli argument to show that
$$\mathbb{P}\left(\lim_{n\to\infty}\sup_{|z|>r}\sum_{k=1}^n |Y_{k,n}(z)| >\epsilon\right)=0.$$

\end{proof}

\subsection{Laurent coefficients}
In the previous section we showed that we could discard the lower order terms of $X_{k,n}(z)$. We now wish to calculate the Laurent coefficients of the remaining higher order terms $\mathcal{X}_{k,n}(z)$ and hence evaluate the fluctuations of the cluster. We first notice that $$\mathbb{E}(\mathcal{X}_{k,n}(z)|\mathcal{F}_{k-1})=0$$ and therefore $\mathcal{X}_{k,n}(z)$ is also a martingale difference array. We aim to use the following result of Mcleish \cite{mcleish1974dependent} to produce a central limit theorem. 
\begin{theorem}[McLeish]\label{6.2}
Let $(X_{k,n})_{1\leq k \leq n}$ be a martingale difference array with respect to the filtration $\mathcal{F}_{k,n}=\sigma(X_{1,n}, X_{2,n},...,X_{k,n}). $ Let $M_{n}=\sum_{i=1}^{n}X_{i,n}$ and assume that;
\begin{enumerate}[\normalfont(I)]
\item for all $\rho>0$, $\sum_{k=1}^n X_{k,n}^2  $ $\mathbb{1}(|X_{k,n}|>\rho)\to 0$ in probability as $n\to \infty$.
\item $\sum_{k=1}^n X_{k,n}^2 \to s^2$ in probability as $n\to\infty$ for some $s^2>0$.
\end{enumerate}
Then $M_{n}$ converges in distribution to $\mathcal{N}(0, s^2).$
\end{theorem}
Note that condition (I) in Theorem \ref{6.2} combines two conditions in \cite{mcleish1974dependent} as a result of the Lindberg condition \cite{lindeberg1922neue}. Theorem \ref{6.2} only applies to real valued random variables and as such we will split $\mathcal{X}_{k,n}(z)$ into real and imaginary parts. We start by calculating the Laurent coefficients.
$$\mathcal{X}_{k,n}(z)=\frac{2c_k^* \sqrt{n}}{e^{-i \theta_k}e^{C_{k+1,n}^{*}}}\left(\frac{1}{1-\frac{1}{e^{-i \theta_k}e^{C_{k+1,n}^{*}}z}}\right).$$
We can choose $|z|>r$ such that $\left|\frac{1}{ e^{-i \theta_k}e^{C_{k+1,n}^{*}}z}\right|<1$, then 
\begin{align*}
\mathcal{X}_{k,n}(z)&=\sum_{m=0}^{\infty}\frac{2c_k^*\sqrt{n}}{( e^{-i \theta_k}e^{C_{k+1,n}^{*}})^{m+1}}\frac{1}{z^m}.
\intertext{So the $m^{\text{th}}$ coefficient is dependent on $n$ and $k$ and we can rewrite $\mathcal{X}_{k, n} $ as}
\mathcal{X}_{k,n}(z)&=\sum_{m=0}^{\infty} a_{k,n}(m)\frac{1}{z^m}
\end{align*}
where $a_{k,n}(m)=\frac{2c_k^* \sqrt{n}}{(e^{C_{k+1,n}^{*}})^{m+1}}e^{i \theta_k(m+1)}.$ So we can calculate real and imaginary parts of these coefficients,
$$\Re(a_{k,n}(m))=\frac{2c_k^*\sqrt{n}}{( e^{C_{k+1,n}^{*}})^{m+1}} \cos(\theta_{k}(m+1)),$$
$$\Im(a_{k,n}(m))=\frac{2c_k^*\sqrt{n}}{( e^{C_{k+1,n}^{*}})^{m+1}} \sin(\theta_{k}(m+1)).$$
In order to use Theorem \ref{6.2} we need to calculate the second moments of the coefficients. We will just consider the case of the real coefficients here but the imaginary coefficients give the same results. Thus, we calculate,
\begin{align*}
\mathbb{E}((\Re(a_{k,n}(m)))^2 | \mathcal{F}_{k-1})&=\frac{4(c_k^*)^2 n }{(e^{C_{k+1,n}^{*}})^{2(m+1)}}\frac{1}{2\pi}\int_{0}^{2\pi} \cos^2(\theta(m+1))d \theta\\
&=\frac{2(c_k^*)^2 n }{( e^{C_{k+1,n}^{*}})^{2(m+1)}}.
\end{align*}
It is clear to see here why we have the same expected value of the imaginary coefficients. So now we can take the sum over $n$,
$$\lim_{n\to \infty} \sum_{k=1}^n \mathbb{E}((\Re(a_{k,n}(m)))^2 | \mathcal{F}_{k-1})=\lim_{n\to \infty}\left(2 n  \sum_{k=1}^n \frac{(c_k^*)^2}{\left(e^{C_{k+1,n}^{*}}\right)^{2(m+1)}}\right).$$
Recall that $c_k^*=\frac{c}{1+\alpha c(k-1)}$ and we have shown we can approximate the term in the \newline denominator in the following way;
$$e^{C_{k+1,n}^{*}}=\left(\frac{1+\alpha c n}{1+\alpha ck}\right)^{\frac{1+\epsilon_{k+1,n}}{\alpha}}$$
where $\epsilon_{k+1,n}$ is the error defined in Lemma \ref{lemma2.2}. Therefore, we can write
$$\lim_{n\to \infty} \sum_{k=1}^n \mathbb{E}((\Re(a_{k,n}(m)))^2 | \mathcal{F}_{k-1})=\lim_{n\to \infty}\left(2 n c^2  \sum_{k=1}^n \frac{\left(1+\alpha ck\right)^{\left(\frac{(1+\epsilon_{k+1,n})(2(m+1))}{\alpha}\right)-2}}{\left(1+\alpha cn\right)^{\left(\frac{(1+\epsilon_{k+1,n})(2(m+1))}{\alpha}\right)}}\right)$$
We know $\epsilon_{k+1,n}\to 0$ so our aim is to show that this term in the sum is insignificant. We define the function $h:\mathbb{R} \to \mathbb{R}$ as the term inside the sum;
$$ h(x):= \frac{\left(1+\alpha ck\right)^{\left(\frac{(1+x)(2(m+1))}{\alpha}\right)-2}}{\left(1+\alpha cn\right)^{\left(\frac{(1+x)(2(m+1))}{\alpha}\right)}}.$$
Our aim is to show, $$\left| \lim_{n\to \infty}2 n c^2  \sum_{k=1}^n \left(h(\epsilon_{k+1,n})-h(0)\right)\right|=0.$$
If we can show this then we can just ignore the $\epsilon_{k,n}$ and find the limit,
$$\lim_{n\to \infty}2 n c^2 \sum_{k=1}^n h(0)$$
which we will show converges to a real number.  We provide this in the form of the following lemma.
\begin{lemma}
With $h: \mathbb{R} \to \mathbb{R}$ defined as above we have $$\left| \lim_{n\to \infty}2 n c^2 \sum_{k=1}^n \left(h(\epsilon_{k+1,n})-h(0)\right)\right|=0$$
\end{lemma}
\begin{proof}
Consider $$|h(\epsilon_{k+1,n})-h(0)|=\left| \frac{\left(1+\alpha ck\right)^{\left(\frac{(1+\epsilon_{k+1,n})(2(m+1))}{\alpha}\right)-2}}{\left(1+\alpha cn\right)^{\left(\frac{(1+\epsilon_{k+1,n})(2(m+1))}{\alpha}\right)}}-\frac{\left(1+\alpha ck \right)^{\left(\frac{(2(m+1))}{\alpha}\right)-2}}{\left(1+\alpha cn\right)^{\left(\frac{(2(m+1))}{\alpha}\right)}}\right|.$$
Then let $y_{k,n}=\left(\frac{1+\alpha c k}{1+\alpha c n }\right)^{\frac{2m+2}{\alpha}}$, thus we can write
$$|h(\epsilon_{k+1,n})-h(0)|=\frac{1}{(1+\alpha c k)^2}|y_{k,n}|\left| y_{k,n}^{\epsilon_{k+1,n}}-1\right|.$$
Furthermore, since $\log(y_{k,n})<1$, 
$$\left| y_{k,n}^{\epsilon_{k+1,n}}-1\right|=\left| e^{\epsilon_{k+1,n} \log y_{k,n}}-1\right|\leq | \epsilon_{k+1,n}||\log y_{k,n}|.$$
So using the first bound on $\epsilon_{k,n}$ from Lemma \ref{lemma2.2} we have,
\begin{align*}
|h(\epsilon_{k+1,n})-h(0)|&\leq \frac{1}{(1+\alpha c k)^2}\left(\frac{1+\alpha c k}{1+\alpha c n }\right)^{\frac{2m+2}{\alpha}}   \frac{\alpha \left(\alpha c^2 (n-k)\right)\left|\log\left(\left(\frac{1+\alpha c k}{1+\alpha c n }\right)^{\frac{2m+2}{\alpha}}\right)\right|}{(1+\alpha ck)(1+\alpha c n)\log\left(\frac{1+\alpha c n}{1+\alpha ck}\right)}\\
&\leq (2m+2)\alpha c^2 n \frac{(1+\alpha c k)^{\frac{2m+2}{\alpha}-3}}{(1+\alpha c n)^{\frac{2m+2}{\alpha}+1}}.
\end{align*}
Now we take the sum over $k$,
\begin{align*}
2 n c^2  \sum_{k=1}^n \left|h(\epsilon_{k+1,n})-h(0)\right|&\leq  4n^2(m+1) \alpha c^4\frac{1}{(1+\alpha c n)^{\frac{2m+2}{\alpha}+1}}\sum_{k=1}^n (1+\alpha c k)^{\frac{2m+2}{\alpha}-3}.
\intertext{Which we can approximate with a Riemann integral;}
2 n c^2  \sum_{k=1}^n \left|h(\epsilon_{k+1,n})-h(0)\right|&\leq 4n^2(m+1) \alpha c^4\frac{1}{(1+\alpha c n)^{\frac{2m+2}{\alpha}+1}}\int_{0}^n (1+\alpha c x)^{\frac{2m+2}{\alpha}-3}dx.
\end{align*}
Now we need to consider cases, firstly in the case where  we have $\frac{2m+2}{\alpha}-3\neq-1$ and so
\begin{align*}&\left| \lim_{n\to \infty}2 n c^2  \sum_{k=1}^n (h(\epsilon_{k+1,n})-h(0))\right|\\
&\leq \lim_{n \to \infty}4n^2(m+1) \alpha c^4\frac{1}{(1+\alpha c n)^{\frac{2m+2}{\alpha}+1}}\left[\frac{1}{\alpha c \left(\frac{2m+2}{\alpha}-2\right)}(1+\alpha c x)^{\frac{2m+2}{\alpha}-2}\right]^n_0\\
&=\lim_{n \to \infty}\left(\frac{2(m+1)\alpha c^3}{ m+1-\alpha}\left(\frac{n^2}{(1+\alpha c n)^3}-\frac{n^2}{(1+\alpha c n)^{\frac{2m+2}{\alpha}+1}}\right)\right).\\
\intertext{Hence, since $0<\alpha<2$,}
&\left| \lim_{n\to \infty}2 n c^2  \sum_{k=1}^n (h(\epsilon_{k+1,n})-h(0))\right|=0
\end{align*}
Now consider the case where $\frac{2m+2}{\alpha}-3=-1$ and so 
\begin{align*}
&\left| \lim_{n\to \infty}2 n c^2 \sum_{k=1}^n( h(\epsilon_{k+1,n})-h(0))\right|
\\&\leq \lim_{n \to \infty}4n^2(m+1) \alpha c^4\frac{1}{(1+\alpha c n)^{\frac{2m+2}{\alpha}+1}}\left[\frac{1}{\alpha c}\log(1+\alpha c x)\right]^n_0\\
&=\lim_{n \to \infty}4n^2  c^3\frac{\log(1+\alpha c n)}{(1+\alpha c n)^3}\\
&=0.
\end{align*}
Therefore in all cases we have 
$$\left| \lim_{n\to \infty}2 n c^2  \sum_{k=1}^n (h(\epsilon_{k+1,n})-h(0))\right|=0.$$
\end{proof} 

Hence by using the above lemma we can ignore the $\epsilon_{k+1,n}$ term in our summation. We now want to check the conditions of Theorem \ref{6.2}. We introduce the notation, 
$$A_m^n=\sum_{k=1}^{n} \Re(a_{k,n}(m)),\;\;\;\;\;\;\;\;\;\;\;B_m^n=\sum_{k=1}^{n} \Im(a_{k,n}(m)).$$
We aim to apply Theorem \ref{6.2} to show joint convergence in distribution of $(A_i^n, B_j^n)_{i,j\geq 1}$ to some multivariate Gaussian distribution. The Cram\'er-Wold Theorem (see for example \cite{durrett_2019}) tells us that it is sufficient to show convergence in distribution of all finite linear combinations of $A_i^n$, $B_j^n$. Therefore, let 
$$\mathfrak{X}_{k,n}=\sum_{i=1}^p\mu_i \Re(a_{k,n}(i))+ \sum_{j=1}^q\nu_j \Im(a_{k,n}(j)) $$ 
for some $1\leq p, q<\infty$ and sequences of scalars $(\mu_i)_{1\leq i \leq p}$, $(\nu_j)_{1\leq j \leq q}$. It follows that $\mathfrak{X}_{k,n}$ is also a martingale difference array. Therefore, we will apply Theorem \ref{6.2} to $\mathfrak{X}_{k,n}$ to show that we have convergence in distribution of finite linear combinations and hence joint convergence in distribution to a multivariate distribution. We start by checking condition (II) of Theorem \ref{6.2} holds. 
\begin{lemma}\label{6.4}
Assume $m>0$ and $0<\alpha<2$. Then
$$\lim_{n\to \infty} \sum_{k=1}^n \mathbb{E}((\Re(a_{k,n}(m)))^2 | \mathcal{F}_{k-1})=\lim_{n\to \infty} \sum_{k=1}^n \mathbb{E}((\Im(a_{k,n}(m)))^2 | \mathcal{F}_{k-1})=\frac{2}{\alpha(2m+2-\alpha)}.$$
Furthermore, for any $m_1,  m_2\geq 1$,
\begin{align*}
\mathrm{Cov}(\Re (a_{k,n}(m_1),\Im (a_{k,n}(m_2)))=0
\end{align*}
and if $m_1\neq m_2$,
$$\mathrm{Cov}(\Re (a_{k,n}(m_1),\Re (a_{k,n}(m_2)))=\mathrm{Cov}(\Im (a_{k,n}(m_1),\Im (a_{k,n}(m_2)))=0.$$

\end{lemma}
\begin{proof}
We have shown above that, in the case of the real coefficients, calculating \newline $\lim_{n\to \infty} \sum_{k=1}^n \mathbb{E}((\Re(a_{k,n}(m)))^2 | \mathcal{F}_{k-1})$ reduces to calculating the expression $$\lim_{n\to \infty}\left(2 n c^2  \sum_{k=1}^n \frac{\left(1+\alpha ck\right)^{\left(\frac{(2(m+1))}{\alpha}\right)-2}}{\left(1+\alpha cn\right)^{\left(\frac{(2(m+1))}{\alpha}\right)}}\right).$$
The imaginary coefficients follow by the same argument. We can approximate this with a Riemann integral.
$$2 n c^2  \sum_{k=1}^n \frac{\left(1+\alpha ck\right)^{\left(\frac{(2(m+1))}{\alpha}\right)-2}}{\left(1+\alpha cn\right)^{\left(\frac{(2(m+1))}{\alpha}\right)}}\approx \frac{2 n c^2}{\left(1+\alpha cn\right)^{\left(\frac{(2(m+1))}{\alpha}\right)}}\int_0^{n}(1+\alpha c x)^{\left(\frac{(2(m+1))}{\alpha}\right)-2} dx$$
Since for all $m>0$ and $0<\alpha<2$, $\frac{(2(m+1))}{\alpha}-2>-1$, we have,
\begin{align*}
&=\frac{2 n c^2}{\left(1+\alpha cn\right)^{\left(\frac{(2(m+1))}{\alpha}\right)}}\left[\frac{1}{2c(m+1)-\alpha c}(1+\alpha c x)^{\left(\frac{(2(m+1))}{\alpha}\right)-1}\right]_0^{n}\\
&=\frac{2 c^2}{2c(m+1)-\alpha c}\left[\frac{n}{(1+\alpha cn)}-\frac{n}{\left(1+\alpha cn\right)^{\left(\frac{(2(m+1))}{\alpha}\right)}}\right].
\end{align*}
We know for  all $m>0$ and $0<\alpha<2$, $\frac{(2(m+1))}{\alpha}>1$ and so when we take the limit as $n\to \infty$ we have,
$$\lim_{n\to \infty}\left(2 n c^2  \sum_{k=1}^n \frac{\left(1+\alpha ck \right)^{\left(\frac{(2(m+1))}{\alpha}\right)-2}}{\left(1+\alpha cn\right)^{\left(\frac{(2(m+1))}{\alpha}\right)}}\right)=\frac{2 }{\alpha(2(m+1)-\alpha )}.$$
Therefore,
$$\lim_{n\to \infty} \sum_{k=1}^n \mathbb{E}((\Re(a_{k,n}(m)))^2 | \mathcal{F}_{k-1})=\lim_{n\to \infty} \sum_{k=1}^n \mathbb{E}((\Im(a_{k,n}(m)))^2 | \mathcal{F}_{k-1})=\frac{2}{\alpha(2m+2-\alpha)}.$$
Furthermore, calculating the covariance pairwise of each combination of the random variables we see that for any $m_1, m_2$
\begin{align*}
\mathrm{Cov}(\Re (a_{k,n}(m_1)),\Im (a_{k,n}(m_2)))&=\mathbb{E}(\Re (a_{k,n}(m_1))\Im (a_{k,n}(m_2)))\\
&=\frac{4n(c_k^{*})^2}{( e^{C_{k+1,n}^{*}})^{m_1+m_2+2}}\int_{0}^{2\pi}\cos(\theta (m_1+1))\sin(\theta (m_2+1)) d\theta\\
&=0.
\end{align*}
Moreover for $m_1\neq m_2$,
\begin{align*}
\mathrm{Cov}(\Re (a_{k,n}(m_1)),\Re (a_{k,n}(m_2)))&=\mathbb{E}(\Re (a_{k,n}(m_1))\Re (a_{k,n}(m_2)))\\
&=\frac{4n(c_k^{*})^2}{( e^{C_{k+1,n}^{*}})^{m_1+m_2+2}}\int_{0}^{2\pi}\cos(\theta (m_1+1))\cos(\theta (m_2+1)) d\theta\\
&=0.
\end{align*}
For $m_1\neq m_2$,
\begin{align*}
\mathrm{Cov}(\Im (a_{k,n}(m_1)),\Im (a_{k,n}(m_2)))&=\mathbb{E}(\Im (a_{k,n}(m_1))\Im (a_{k,n}(m_2)))\\
&=\frac{4n(c_k^{*})^2}{( e^{C_{k+1,n}^{*}})^{m_1+m_2+2}}\int_{0}^{2\pi}\sin(\theta (m_1+1))\sin(\theta (m_2+1)) d\theta\\
&=0.
\end{align*}

\end{proof}
So we have shown that sum of the second moments of the real and imaginary parts converge. Note that it is clear to see that letting $\alpha=2$ will not provide a finite limit using the above lemma. To apply Theorem \ref{6.2} we need to show that $\sum_{k=1}^n (\mathfrak{X}_{k,n})^2 $ also converges. We prove this with the following lemma, using a similar method to that of Silvestri in \cite{silvestri2017fluctuation}. 

\begin{lemma}\label{b}
Let $0<\alpha<2$ and assume for each $m>0$, the following limit holds in probability for some $s^2>0$, $$\lim_{n\to \infty} \sum_{k=1}^n \mathbb{E}((\Re(a_{k,n}(m)))^2 | \mathcal{F}_{k-1})=\lim_{n\to \infty} \sum_{k=1}^n \mathbb{E}((\Im(a_{k,n}(m)))^2 | \mathcal{F}_{k-1})=s^2.$$ 
Then for each $m>0$, in probability,
$$\lim_{n\to \infty} \sum_{k=1}^n \Re(a_{k,n}(m)))^2 =\lim_{n\to \infty} \sum_{k=1}^n \Im(a_{k,n}(m)))^2=s^2.$$ 
Therefore, if the following limit holds in probability for some $s^2>0$, $$\lim_{n\to \infty} \sum_{k=1}^n \mathbb{E}((\mathfrak{X}_{k,n})^2 | \mathcal{F}_{k-1})=s^2$$
then in probability,
$$\lim_{n\to \infty} \sum_{k=1}^n(\mathfrak{X}_{k,n})^2 =s^2.$$
\end{lemma}
\begin{proof}
First we note that  $$\mathcal{Y}_k(z)=(\Re(a_{k,n}(m)))^2-\mathbb{E}((\Re(a_{k,n}(m)))^2 | \mathcal{F}_{k-1})$$ is a martingale difference array with respect to the filtration $(\mathcal{F}_{k, n})_{k\leq n}$. We need to show $\mathbb{P}(|\sum_{k=1}^n\mathcal{Y}_k(z)|>\eta)\to 0 $ as $n\to \infty$. So we first notice that by Markov's inequality,
$$\mathbb{P}\left(|\sum_{k=1}^n\mathcal{Y}_k|>\eta \right)\leq \frac{1}{\eta^2}\mathbb{E}\left(|\sum_{k=1}^n\mathcal{Y}_k|^2\right)=\frac{1}{\eta^2}\sum_{k=1}^n\mathbb{E}(\mathcal{Y}_k^2).$$
and so finally by using the property that for a random variable $X$, $\mathbb{E}((X-\mathbb{E}(X))^2)\leq \mathbb{E}(X^2)$ we see 
$$\mathbb{P}\left(|\sum_{k=1}^n\mathcal{Y}_k|>\eta\right)\leq \frac{1}{\eta^2}\sum_{k=1}^n \mathbb{E}(\Re(a_{k,n}(m)))^4).$$
We have shown,
$$\Re(a_{k,n}(m))=\frac{2 c_k^*\sqrt{n}}{( e^{C_{k+1,n}^{*}})^{m+1}}\cos(\theta_k(m+1)).$$
So using the property that $c_k^{*}=\frac{c}{1+\alpha c(k-1)}$ and $ e^{-C_{k+1,n}^{*}}\leq \left(\frac{1+\alpha c k}{1+\alpha c n}\right)^{1/\alpha}$ we reach the upper bound,
\begin{equation}\label{eq:1}
\Re(a_{k,n}(m))\leq 2c(1+\alpha c)\sqrt{n}\frac{(1+\alpha c k)^{\frac{m+1}{\alpha}-1}}{(1+\alpha c n)^{\frac{m+1}{\alpha}}}.
\end{equation}
Thus, 
$$\Re(a_{k,n}(m))^4\leq \left(2c(1+\alpha c)\right)^4 \frac{n^2(1+\alpha c k)^{\frac{4(m+1)}{\alpha}-4}}{(1+\alpha c n)^{\frac{4(m+1)}{\alpha}}}.$$
Then we consider cases. If $0<\alpha\leq\frac{4}{3}(m+1)$ then when we sum over $k$ we reach the following bound,
$$\frac{1}{\eta^2}\left(\sum_{k=1}^n \mathbb{E}\left((\Re(a_{k,n}(m)))^4\right)\right)\leq \lambda(\alpha,c)  \frac{1}{n}$$
where $\lambda(\alpha,c)$  is some constant. This converges to zero as $n \to \infty$. Moreover, if $\frac{4}{3}(m+1)<\alpha<2$ then when we sum over $k$ we reach the following bound,
$$\frac{1}{\eta^2}\left(\sum_{k=1}^n \mathbb{E}\left((\Re(a_{k,n}(m)))^4\right)\right)\leq \lambda(\alpha,c)   \frac{n}{n^{\frac{4(m+1)}{\alpha}}}$$
where $\lambda(\alpha,c)$  is some constant. This converges to zero as $n \to \infty$. Therefore in both cases we have convergence to zero. The proof of the imaginary case holds by the same argument. Now we consider $\lim_{n\to \infty} \sum_{k=1}^n \mathbb{E}((\mathfrak{X}_{k,n})^2 | \mathcal{F}_{k-1})$. By the same argument as above, $$\mathbb{P}\left(\left|\sum_{k=1}^n \left((\mathfrak{X}_{k,n})^2-\mathbb{E}((\mathfrak{X}_{k,n})^2 | \mathcal{F}_{k-1})\right)\right|>\eta\right)\leq \frac{1}{\eta^2}\sum_{k=1}^n \mathbb{E}((\mathfrak{X}_{k,n})^4).$$
Since the function $f(x)=x^4$, where $f:\mathbb{R}\to\mathbb{R}$, is convex, by Jensen's inequality,
\begin{align*}
(\mathfrak{X}_{k,n})^4&\leq\frac{\sum_{i=1}^p|\mu_i|(\Re (a_{k,n}(i))^4+\sum_{j=1}^q|\nu_j|(\Im (a_{k,n}(j))^4}{\sum_{i=1}^p|\mu_i|+\sum_{j=1}^q|\nu_j|}.
\end{align*}
Therefore,
\begin{align*}
&\sum_{k=1}^n \mathbb{E}((\mathfrak{X}_{k,n})^4)\\
&\leq \sum_{k=1}^n\left( \frac{\sum_{i=1}^p|\mu_i|\mathbb{E}\left((\Re (a_{k,n}(i))^4\right)+\sum_{j=1}^q|\nu_j|\mathbb{E}\left((\Im (a_{k,n}(j))^4\right)}{p\inf_{1\leq i\leq p}|\mu_i|+q \inf_{1\leq j\leq q}|\nu_j|}\right)\\
&\frac{\leq p\sup_{1\leq i\leq p}\left(|\mu_i|\sum_{k=1}^n \mathbb{E}\left((\Re (a_{k,n}(i))^4\right)\right)+q\sup_{1\leq j\leq q}\left(|\nu_j|\sum_{k=1}^n \mathbb{E}\left((\Re (a_{k,n}(j))^4\right)\right)}{{p\inf_{1\leq i\leq p}|\mu_i|+q \inf_{1\leq j\leq q}|\nu_j|}}\\
&\to 0
\end{align*}
as $n\to \infty$ by above. 
\end{proof}
Therefore, we have shown, in the form of the following corollary, that the condition (II) of Theorem \ref{6.2} is satisfied. 
\begin{corollary}\label{6.6}
For $a_{k,n}(m)$ defined as above, then for each $m\geq 0$ the following limit holds in probability,
$$\lim_{n\to \infty} \sum_{k=1}^n \Re(a_{k,n}(m)))^2 =\lim_{n\to \infty} \sum_{k=1}^n \Im(a_{k,n}(m)))^2=\frac{2}{\alpha(2m+2-\alpha)}.$$
Therefore, with $\mathfrak{X}_{k,n}$ defined as above, 
$$\lim_{n\to \infty} \sum_{k=1}^n(\mathfrak{X}_{k,n})^2=\sum_{i=1}^p\left(\mu_i^2 \frac{2}{\alpha(2i+2-\alpha)}\right)+ \sum_{j=1}^q\left(\nu_j^2 \frac{2}{\alpha(2j+2-\alpha)}\right)$$
\end{corollary}
So now we just need show condition (I) of Theorem \ref{6.2} holds in order to apply it. We will again use a similar method to Silvestri \cite{silvestri2017fluctuation}.
\begin{lemma}\label{6.7}
Let $0<\alpha<2$ and let $\mathfrak{X}_{k,n}$ be defined as above. Let $\rho>0$ then the following limit holds in probability,
$$ \sum_{k=1}^n (\mathfrak{X}_{k,n})^2 \mathbb{1}(|\mathfrak{X}_{k,n}|>\rho)\to 0$$
as $n\to \infty$.
\end{lemma}
\begin{proof}
We use a similar method as \cite{silvestri2017fluctuation}. Let $\delta>0$ then 
\begin{align*}
&\mathbb{P}\left(\sum_{k=1}^n (\mathfrak{X}_{k,n})^2 \mathbb{1}(|\mathfrak{X}_{k,n}|>\rho)>\delta\right)\\
&\leq \mathbb{P}\left(\max_{1\leq k\leq n}|\mathfrak{X}_{k,n}|>\rho\right)\\
&\leq \frac{1}{\rho}\mathbb{E}\left(\max_{1\leq k\leq n}|\mathfrak{X}_{k,n}|\right)\\
&\leq \frac{1}{\rho}\left(\sum_{i=1}^p\mu_i \mathbb{E}\left(\max_{1\leq k\leq n}|\Re(a_{k,n}(i))|\right)+ \sum_{j=1}^q\nu_j \mathbb{E}\left(\max_{1\leq k\leq n}|\Im(a_{k,n}(j))|\right)\right)
\end{align*}
with the second inequality following by Markov's inequality. 
As in the proof of Lemma \ref{b}, we have shown that for each $m\geq 0$,
\begin{align*}
&\left|\Re(a_{k,n}(m))\right|\leq 2c(1+\alpha c)\sqrt{n}\frac{(1+\alpha c k)^{\frac{m+1}{\alpha}-1}}{(1+\alpha c n)^{\frac{m+1}{\alpha}}}.
\intertext{So if $m+1\geq \alpha$,}
&\max_{0\leq k\leq n}\left|\Re(a_{k,n}(m))\right|\leq 2c(1+\alpha c)\sqrt{n}\frac{1}{(1+\alpha c n)}.
\intertext{Then if $m+1<\alpha$,}
&\max_{0\leq k\leq n}\left|\Re(a_{k,n}(m))\right|\leq 2c(1+\alpha c)\sqrt{n}\frac{1}{(1+\alpha c n)^{\frac{m+1}{\alpha}}}.
\end{align*}
In both cases $\max_{0\leq k\leq n}\Re(a_{k,n}(m))$ converges to zero as $n \to \infty$. The imaginary case follows by the same argument. Thus the finite sums also converge to zero,
$$\frac{1}{\rho}\left(\sum_{i=1}^p\mu_i \mathbb{E}\left(\max_{1\leq k\leq n}|\Re(a_{k,n}(i))|\right)+ \sum_{j=1}^q\nu_j \mathbb{E}\left(\max_{1\leq k\leq n}|\Im(a_{k,n}(j))|\right)\right)\to 0$$
 as $n\to \infty$. Therefore, 
$$ \sum_{k=1}^n (\mathfrak{X}_{k,n})^2 \mathbb{1}(|\mathfrak{X}_{k,n}|>\rho)\to 0$$
in probability as $n\to \infty$.
\end{proof}
So now we have all we need in order to apply Theorem \ref{6.2}. This leads to the following result.
\begin{theorem}\label{6.8}
Let $0<\alpha<2$ and $A^n_m$, $B_m^n$ defined as above. Then the following limit holds in joint distibution, 
\begin{align*}    
        \begin{pmatrix}      
           A^n_1+iB_1^n \\           
           \vdots \\
           A^n_m+iB_m^n \\
           \vdots \\
    \end{pmatrix}
    \to
   \begin{pmatrix}      
           A_1 \\           
           \vdots \\
           A_m \\
           \vdots \\
    \end{pmatrix}
    +i
    \begin{pmatrix}      
           B_1 \\           
           \vdots \\
           B_m \\
           \vdots \\
    \end{pmatrix}
  \end{align*}
where $(A_i,B_j)_{i,j\geq 0}$ is a multivariate Gaussian distribution with $\mathbb{E}(A_i)=\mathbb{E}(B_j)=0$ for all $i,j\geq 0$ and covariance structure given by,
\begin{align*}
\mathrm{Cov}(A_i,B_j)&=0 \\
\mathrm{Cov}(A_i,A_j)=\mathrm{Cov}(B_i,B_j)&=\delta_{i,j}\left(\frac{2}{\alpha(2i+2-\alpha)}\right)
\end{align*}
for any $i,j\geq0$ where $\delta_{i,j}$ is the Kronecker delta function.
\end{theorem}

\subsection{Convergence as a holomorphic function}
Now that we have proved that the Laurent coefficients converge, we wish to show that we also have the convergence of the fluctuations as a holomorphic function. We first define the functions,
$$\tilde{\mathcal{F}}(n,z)=\sqrt{n}\left(e^{-C_{1,n}^{*}}\phi_{n}(z)-z\right)$$
and 
$$\mathcal{F}(z)=\sum_{m=0}^{\infty}(A_m+iB_m)z^{-m}$$
where  $A_m$, $B_m$ are defined as in Theorem \ref{6.8}.
Our aim is to show that $\tilde{\mathcal{F}}(n,z)\to \mathcal{F}(z)$ in distribution as $n \to \infty$ on the space of holomorphic functions, $\mathcal{H}$, equipped with the metric,
$$\textbf{d}_{\mathcal{H}}(f,g)=\sum_{m\geq 0} 2^{-m}\left(1 \wedge \sup_{|z|\geq 1+2^{-m}}|f(z)-g(z)|\right).$$
We use a similar method as in \cite{norris2019scaling} by defining,
$$\textbf{d}_r(f,g)=\sup_{|z|>r}|f(z)-g(z)|.$$  To make notation easier, we also define $M(n,m)=\sum_{k=1}^{n} a_{k,n}(m)$. We first need the following lemma used to discard the tail terms.

\begin{lemma}\label{lemma12}
Let $r>1$ and $N>0$ then for any $\epsilon>0$
$$\lim_{T\to\infty}\sup_{n>N}\mathbb{P}\left(\textbf{d}_r\left(\sum_{m=T}^{\infty}M(n,m)z^{-m}, 0\right)>\epsilon\right)=0.$$
\begin{proof}
Using the definition of $\textbf{d}_r(f,g)$ we see that, 
$$\textbf{d}_r\left(\sum_{m=T}^{\infty}M(n,m)z^{-m}, 0\right)=\sup_{|z|>r}\left|\sum_{m=T}^{\infty}M(n,m)z^{-m}\right|.$$
By Markov's inequality,
\begin{align*}
\mathbb{P}\left(\textbf{d}_r\left(\sum_{m=T}^{\infty}M(n,m)z^{-m}, 0\right)>\epsilon\right)&\leq \frac{1}{\epsilon^2}\mathbb{E}\left( \sup_{|z|>r}\left|\sum_{m=T}^{\infty}M(n,m)z^{-m}\right|^2\right)\\
&\leq \frac{1}{\epsilon^2}\mathbb{E}\left( \sup_{|z|>r}\left(\sum_{m=T}^{\infty}|M(n,m)||z|^{-m}\right)^2\right)\\
&\leq \frac{1}{\epsilon^2}\mathbb{E}\left( \left(\sum_{m=T}^{\infty}|M(n,m)|r^{-m}\right)^2\right).
\intertext{Using the Cauchy-Schwartz inequality we have,}
\mathbb{P}\left(\textbf{d}_r\left(\sum_{m=T}^{\infty}M(n,m)z^{-m}, 0\right)>\epsilon\right)&\leq \frac{1}{\epsilon^2}\mathbb{E}\left( \left(\sum_{m=T}^{\infty}|M(n,m)|^2 r^{-m}\right)\left(\sum_{m=T}^{\infty} r^{-m}\right)\right)\\
&\leq \frac{\lambda(r)}{\epsilon^2}\mathbb{E} \left(\sum_{m=T}^{\infty}|M(n,m)|^2 r^{-m}\right)
\intertext{where $\lambda(r)$ is some constant dependent on r. Then we can take the expectation inside the sum, thus, }
\mathbb{P}\left(\textbf{d}_r\left(\sum_{m=T}^{\infty}M(n,m)z^{-m}, 0\right)>\epsilon\right)&\leq \frac{1}{\epsilon^2} \sum_{m=T}^{\infty}\mathbb{E}\left( |M(n,m)|^2 \right)r^{-m} .
\end{align*}
Now notice that,
\begin{align*}
\mathbb{E}\left(|M(n,m)|^2\right)=\mathbb{E}\left(\left|\sum_{k=1}^n a_{k,n}(m)\right|^2\right)\leq \mathbb{E}\left(\sum_{k=1}^n (\Re(a_{k,n}(m))^2+(\Im(a_{k,n}(m))^2\right).
\end{align*}
But in equation (\ref{eq:1}) we show that 
\begin{align*}
(\Re(a_{k,n}(m))^2+(\Im(a_{k,n}(m))^2\leq \lambda(\alpha,c,r) n^{1-\frac{2(m+1)}{\alpha}}k^{\frac{2(m+1)}{\alpha}-2}
\end{align*}
where $\lambda(\alpha,c,r)$ is some constant. Taking the sum over $k$ we see that,
$$\mathbb{E}\left(|M(n,m)|^2\right)\leq \lambda(\alpha,c,r).$$
Therefore, 
\begin{align*}
\lim_{T\to\infty}\sup_{n>N}\mathbb{P}\left(\textbf{d}_r\left(\sum_{m=T}^{\infty}M(n,m)z^{-m}, 0\right)>\epsilon\right)&\leq \lim_{T\to \infty}\frac{1}{\epsilon^2}\lambda(\alpha,c,r) \sum_{m=T}^{\infty}r^{-m}\\
&\to 0 \text{ as }T \to \infty.
\end{align*}
\end{proof}
\end{lemma}
Therefore, through Theorem \ref{6.8} we have shown that we have convergence of the Laurent coefficients. Moreover, Lemma \ref{lemma12} shows that the tails of the Laurent series tend to zero in the limit. We can then combine these two results with a result of Billingsley \cite{billingsley1999convergence} to show that we have convergence as a holomorphic function and therefore the fluctuations behave like a Gaussian field.
\newpage
\begin{thm3}
Let $0<\alpha<2$ and $\phi_n$ be defined as in Theorem \ref{a}. Then as $n\to \infty$,
$$\sqrt{n}\left(e^{-\sum_{i=1}^{n}c_i^{*}}\phi_{n}(z)-z\right)\to \mathcal{F}(z)$$
in distribution on $\mathcal{H}$, where $\mathcal{H}$ is the space of holomorphic functions on ${|z| > 1}$,  equipped with metric $\textbf{d}_{\mathcal{H}}$ defined above, and where 
$$\mathcal{F}(z)=\sum_{m=0}^{\infty}(A_m+iB_m)z^{-m}$$
and $A_m$, $B_m \sim  \mathcal{N}\left(0, \frac{2}{\alpha(2m+2-\alpha)}\right)$ and $A_m$, $B_k$ independent for all choices of $m$ and $k$.
\end{thm3}
\section*{Acknowledgements}
We would like to thank Steffen Rhode for many useful conversations on his previous work in this area. We would also like to thank Vincent Beffara for a helpful discussion leading to the proof of Theorem 1.1. Finally, we would like to thank the anonymous referee for a careful reading of our paper and many useful comments. 
\bibliographystyle{unsrt}
\newpage

\bibliography{newbib}

\begin{thebibliography}{10}

\bibitem{eden}
Murray Eden.
\newblock A two-dimensional growth process.
\newblock In {\em Proc. 4th {B}erkeley {S}ympos. {M}ath. {S}tatist. and
  {P}rob., {V}ol. {IV}}, pages 223--239. Univ. California Press, Berkeley,
  Calif., 1961.

\bibitem{dla}
T.~A. Witten and L.~M. Sander.
\newblock Diffusion-{L}imited {A}ggregation.
\newblock {\em Phys. Rev. B}, 27:5686--5697, May 1983.

\bibitem{rohde2005some}
Steffen Rohde and Michel Zinsmeister.
\newblock Some remarks on {L}aplacian growth.
\newblock {\em Topology and its Applications}, 152(1-2):26--43, 2005.

\bibitem{hastings1998laplacian}
Matthew~B Hastings and Leonid~S Levitov.
\newblock Laplacian growth as one-dimensional turbulence.
\newblock {\em Physica D: Nonlinear Phenomena}, 116(1-2):244--252, 1998.

\bibitem{norris2019scaling}
James Norris, Vittoria Silvestri, and Amanda Turner.
\newblock Scaling limits for planar aggregation with subcritical fluctuations.
\newblock {\em arXiv preprint arXiv:1902.01376}, 2019.

\bibitem{norris2012hastings}
James Norris and Amanda Turner.
\newblock Hastings--{L}evitov aggregation in the small-particle limit.
\newblock {\em Communications in Mathematical Physics}, 316(3):809--841, 2012.

\bibitem{viklund2015small}
Fredrik~Johansson Viklund, Alan Sola, and Amanda Turner.
\newblock Small-particle limits in a regularized {L}aplacian random growth
  model.
\newblock {\em Communications in Mathematical Physics}, 334(1):331--366, 2015.

\bibitem{silvestri2017fluctuation}
Vittoria Silvestri.
\newblock Fluctuation results for {H}astings--{L}evitov planar growth.
\newblock {\em Probability Theory and Related Fields}, 167(1-2):417--460, 2017.

\bibitem{pommerenke1975univalent}
Christian Pommerenke.
\newblock Univalent functions.
\newblock {\em Vandenhoeck and Ruprecht}, 1975.

\bibitem{freedman1975tail}
David~A Freedman.
\newblock On tail probabilities for martingales.
\newblock {\em the Annals of Probability}, pages 100--118, 1975.

\bibitem{mcleish1974dependent}
Donald~L McLeish et~al.
\newblock Dependent central limit theorems and invariance principles.
\newblock {\em the Annals of Probability}, 2(4):620--628, 1974.

\bibitem{lindeberg1922neue}
Jarl~Waldemar Lindeberg.
\newblock Eine neue herleitung des exponentialgesetzes in der
  wahrscheinlichkeitsrechnung.
\newblock {\em Mathematische Zeitschrift}, 15(1):211--225, 1922.

\bibitem{durrett_2019}
Rick Durrett.
\newblock {\em Probability: Theory and Examples}.
\newblock Cambridge Series in Statistical and Probabilistic Mathematics.
  Cambridge University Press, fifth edition, 2019.

\bibitem{billingsley1999convergence}
Patrick Billingsley.
\newblock Convergence of probability measures {J}ohn {W}iley \& {S}ons.
\newblock {\em INC, New York}, 2(2.4), 1999.

\end{thebibliography}

\end{document}